\documentclass[a4paper,reqno]{amsart}
\usepackage{a4wide}

\usepackage[disable]{todonotes}

\usepackage{enumerate}
\usepackage{easybmat}

\usepackage{hyperref}
\usepackage{amsmath}
\usepackage{t1enc}
\def\frak{\mathfrak}

\let\e=\varepsilon
\let\i=\iota

\def\E{\mathbb{E}}

\def\R{\mathbb{R}}

\def\cT{\mathcal{T}}

\def\cS{\mathcal{S}}

\def\al{\alpha}
\def\be{\beta}
\def\ga{\gamma}
\def\de{\delta}
\def\ep{\varepsilon}
\def\vo{\epsilon}

\def\ka{\kappa}

\def\si{\sigma}

\def\Ga{\Gamma}
\def\De{\Delta}
\def\Th{\Theta}
\def\La{\Lambda}

\def\Ph{\Phi}

\def\Om{\Omega}
\def\Up{\Upsilon}
\def\na{\nabla}

\newcommand{\wh}{\widehat}
\newcommand{\wt}{\widetilde}

\newcommand{\newc}{\newcommand}

\newtheorem{theorem}{Theorem}[section]
\newtheorem{lemma}[theorem]{Lemma}
\newtheorem{proposition}[theorem]{Proposition}

\theoremstyle{remark}
\newtheorem{remark}[theorem]{\rm\bf Remark}
\newtheorem{remarks}[theorem]{\rm\bf Remarks}
\newtheorem*{remark*}{\rm\bf Remark}

\usepackage{amssymb,stmaryrd}
\usepackage{amscd}

\newcommand{\Rho}{{\mbox{\sf P}}}

\newcommand{\sfW}{{\mbox{\sf W}}}

\newc{\aR}{\mbox{\boldmath{$ R$}}}
\newc{\aS}{\mbox{\boldmath{$ S$}}}
\newc{\aDeR}{\mbox{\boldmath{$ U$}}_B{}^P{}_C{}^Q}
\newc{\aDe}{\mbox{\boldmath$ \Delta$}}
\newc{\aNd}{\mbox{\boldmath$ \nabla$}}

\newc{\aK}{\mbox{\boldmath{$ K$}}}
\newc{\aL}{\mbox{\boldmath{$ L$}}}

\def\sideremark#1{\ifvmode\leavevmode\fi\vadjust{\vbox to0pt{\vss
 \hbox to 0pt{\hskip\hsize\hskip1em
 \vbox{\hsize3cm\tiny\raggedright\pretolerance10000
 \noindent #1\hfill}\hss}\vbox to8pt{\vfil}\vss}}}%
                        
\SetSymbolFont{stmry}{bold}{U}{stmry}{m}{n}

\let\E=\cE
\let\ti=\tilde
\let\bs=\boldsymbol
\def\tnabla{\bs{\nabla}}
\def\T{\bs\cT}

\def\pmat#1{\begin{pmatrix}#1\end{pmatrix}}
\def\smat#1{\left(\begin{smallmatrix}#1\end{smallmatrix}\right)}
\def\dd#1#2{\ifx#2\tfrac{d}{d#1}\else\tfrac{d^{#2}}{d#1^{#2}}\fi}
\def\ddt#1{\dd{t}{#1}}
\def\dds#1{\dd{s}{#1}}
\def\spann#1{\langle{#1}\rangle}
\def\spanb#1{\big\langle{#1}\big\rangle}

\def\dif#1#2{#1^{(#2)}{}}

\def\form#1{\bs{#1}}
\def\dform#1{\form{#1}'}
\def\ddform#1{\form{#1}''}
\def\dddform#1{\form{#1}'''}
\def\ddddform#1{\form{#1}''''}
\def\diform#1#2{\form{#1}^{(#2)}}

\def\tform#1{\wt{\form{#1}}}
\def\tdform#1{\tform{#1}{}'}
\def\tddform#1{\tform{#1}{}''}

\def\tdiform#1#2{\tform{#1}^{(#2)}}

\def\g#1#2{#1\bs\cdot#2}

\begin{document}
\title{Conformal theory of curves with tractors}

\author{Josef \v Silhan}
\address{Institute of Mathematics and Statistics \\ Masaryk University \\ Kotl\'a\v{r}sk\'a 2 \\ 61137 Brno\\ Czech Republic} 
\email{silhan@math.muni.cz}

\author{Vojt\v{e}ch \v{Z}\'adn\'\i k}
\address{Faculty of Education \\ Masaryk University \\ Po\v{r}\'\i\v{c}\'\i\ 31 \\ 60300 Brno \\ Czech Republic} 
\email{zadnik@mail.muni.cz}

\date{\today}


\begin{abstract}
We present the general theory of curves in conformal  geometry using tractor calculus.
This primarily involves a tractorial determination of distinguished parametrizations and relative and absolute conformal invariants of generic curves.
The absolute conformal invariants are defined via a tractor analogue of the classical Frenet frame construction and then expressed in terms of relative ones.
This approach applies likewise to conformal structures of any signature;
in the case of indefinite signature we focus especially on the null curves.
It also provides a conceptual tool for handling distinguished families of curves (conformal circles and conformal null helices) and conserved quantities along them.
\end{abstract}

\maketitle

\section{Introduction}\label{Intro}
The local geometry of curves is a classical subject, nowadays developed for various geometric structures using various views and methods.
The study generally starts in flat spaces,
the passage to general curved cases demands new ideas and attitudes. 
We are going to evolve an approach based on conformal tractor calculus.
In this section, we present main sources of inspiration, summarize main results and introduce main tools needed later.

\subsection{Main sources and methods}
To our knowledge, the first general treatment of curves in conformal Riemannian manifolds is due to Fialkow \cite{Fialkow1942}, which is therefore used as a basic reference for comparisons.
In that paper, a conformally invariant Frenet-like approach is developed. 
This primarily requires a natural distinguished parametrization of the curve (in the place of the arc-length parameter in the Riemannian setting), a natural starting object (in the place of the unit tangent vector) and a notion of derivative along the curve (in the place of the restricted Levi-Civita connection).
Having all these instruments, it is easy to derive a conformal analogue of Fernet formulas with a distinguished set of conformal invariants, the conformal curvatures of the curve.
For generic curves in an $n$-dimensional manifold, this yields $n-1$ conformal curvatures, one of which has an exceptional flavour.

The Frenet frame can be seen as an example of a more general concept of moving frame by Cartan.
An application of the latter method for immersed submanifolds, especially curves, in the homogeneous M\"obius space is presented by Schiemangk and Sulanke in \cite{Sulanke1980} and \cite{Sulanke1981}, which we adduce as important sources of inspiration.
That way, the curve in the homogeneous space is covered by a curve in the principal group so that the restriction of the Maurer--Cartan form to the lift reveals the generating set of invariants.
The lift can be identified with, and practically is constructed as, a curve of (pseudo-)orthonormal frames in the ambient Minkowski space, whose dimension is two more higher than the dimension of the M\"obius space.

Both the homogeneous principal bundle with the Maurer--Cartan form and the associated homogeneous vector bundle, whose fibre is the ambient Minkowski space with the induced linear connection, has a counterpart over general conformal manifolds: 
the former leads to the notion of conformal Cartan connection, the latter to the Thomas standard tractor bundle with its linear connection and parallel bundle metric.
These two approaches are basically equivalent, in this paper we exploit the latter one.
One of its main advantages is that 
everything can be set off pretty directly in terms of underlying data, while it still has very conceptual flavour.
This, on the one hand, allows one to easily adapt key ideas from the homogeneous setting to the tractorial one.
On the other hand, expanding any tractorial formula yields a very concrete tensorial expression that allows comparisons to results obtained by that means.
We refer to the essential work \cite{Bailey1994} by Bailey, Eastwood and Gover both for 
the generalities on  tractor calculus and for an initial step in  the study of (distinguished) curves in this manner.

Besides the just cited main sources, there is a wide literature discussing curves in conformal and other geometries from various aspects.
A short review of the development of the topic for conformal structures can be found in the introduction of \cite{Cairns1994}.
Another references that are close to our purposes are \cite{Beffa2003}, \cite{Burstall2010} and \cite{Herzlich2012}.
An exhibition of tensorial techniques in the study of curves in related geometries can be found in \cite{Hlavaty1934}.
For typical subtleties in dealing with null curves in pseudo-Riemannian geometry, see e.g. \cite{Duggal2007}.
For an application of the Cartan's method of moving frame to curves in a broad class of homogeneous spaces, see \cite{Doubrov2013}.

\subsection{Aim, structure and results}
We study the local differential geometry of curves in general conformal manifolds of general signature using the tractor calculus.
Although the subject is indeed classic, the systematic tractorial approach is novel.
This approach has not only obvious formal benefits, but also a potential for discovering new results and relations.
In particular, the discussion for null curves in the case of indefinite signature (for which almost nothing is known) is very parallel to the one in positive definite case and leads to results that we, at least, would never obtain by other means.

The rough structure of the paper is as follows:
In sections \ref{Relative} and \ref{Absolute} we develop the theory for conformal Riemannian structures, in section \ref{Indefinite} we discuss its analogies in indefinite signature.
We primarily deal with generic curves, whereas special cases are briefly mentioned in accompanying remarks.
However, the most special case---\textit{conformal circles}---is dealt individually in section \ref{Geod}.

In section \ref{Relative} we follow the setting of \cite{Bailey1994}, where one already finds the tractorial determination of the preferred class of \textit{projective} parametrizations on any curve 
as well as  the projectively parametrized {conformal circles}.
Continuing further with tractors of higher order, we build a natural reservoir of relative conformal invariants associated to any curve (subsection \ref{relative}).
The simplest of these leads to the notion of \textit{conformal arc-length}, the distinguished conformally invariant parametrization of the curve (subsection \ref{arc}).
This invariant vanishes identically along the curve if and only if the curve is an arbitrarily parametrized {conformal circle} (Proposition \ref{prop-geod}).
As a demonstration of the ease of use of tractors we describe some conserved quantities along projectively parametrized conformal circles on manifolds admitting an almost Einstein scale, respectively normal conformal Killing field (subsection \ref{conserve}).

In section \ref{Absolute} we launch the Frenet-like procedure to absolute conformal invariants of curve:
\begin{enumerate}[$\bs\cdot$]
\item construct a pseudo-orthonormal tractor Frenet frame along the given curve,
\item differentiate with respect to the conformal arc-length and extract the tractor Frenet formulas,
\item the coefficients of that system determine the generating set of invariants.
\end{enumerate}
For the curve in an $n$-dimensional manifold, we have $n-1$ conformal curvatures, one of which has an exceptional flavour.
The vanishing of the exceptional curvature has an immediate interpretation relating the above mentioned parametrizations (Proposition \ref{prop-K1}).
By construction, all conformal curvatures are expressed via the tractors from the tractor Frenet frame, i.e. with respect to the conformal arc-length parametrization.
Alternative expressions in terms of initial tractors are also possible.
In particular, for all nonexceptional curvatures, we have very simple formulas using the previously defined relative conformal invariants, i.e. with respect to an arbitrary parametrization (Theorem \ref{husty-thm}).
For a given scale, 
we also indicate how to express conformal curvatures in terms of the Riemannian ones (subsection \ref{riemann}).
Besides the pure pleasure, this effort allows a comparison of our invariants with those in the literature, especially in \cite{Fialkow1942}.

In section \ref{Indefinite} we adapt the previous scheme to conformal manifolds of indefinite signature.
In that case we distinguish space-, time- and light-like curves according to the type of their tangent vectors.
A full type classification of curves (via associated tractors) becomes very rich, depending on the dimension and signature.
The construction of the tractor Frenet frame has to be adapted to the respective type, which makes any attempt on its universal description impossible.
Avoiding the complicated branching of the discussion, we only point out the main features (Remark \ref{indef4}(3)) and focus on the light-like curves.
Among these curves we identify an appropriate analogue of conformal circles, the \textit{conformal null helices}.
Notably, in their characterization another family of relative conformal  invariants of Wilczynski type appears (Theorem \ref{r-null}).
More details on the construction of the tractor Frenet frame and expressions of the corresponding conformal curvatures are discussed in the case of Lorenzian signature (subsection \ref{lorenz}).

\subsection{Notation and conventions}
Most of the following conventions is taken from \cite{Bailey1994}.
A \textit{conformal structure} of signature $(p,q)$ on a smooth manifold $M$ of dimension $n=p+q$ is a class of pseudo-Riemannian metrics of signature $(p,q)$ that differ by a multiple of an  everywhere positive function.
For all tensorial objects on  $M$ we use the standard abstract index notation.
Thus, the symbol $\mu^a$ and $\mu_a$ refers to a section of the tangent and cotangent bundle, which is denoted as $\E^a:=TM$ and $\E_a:=T^*M$, respectively, multiple indices denote tensor products, e.g.\
$\mu_a{}^b$ is a section of $\E_a{}^b:=T^*M\otimes TM$ etc.
Round brackets denote symmetrization and square brackets denote skew symmetrization of enclosed
indices, e.g.\ sections  of $\E_{[ab]} = T^*M \wedge T^*M$ are 2-forms on $M$.
By $\E[w]$ we denote the density bundle of \textit{conformal weight} $w$, which is just the bundle of ordinary 
$(-\frac{w}{n})$-densities.
Tensor products with another bundles are denoted as $\E^a[w]:=\E^a\otimes\E[w]$ etc.
In what follows, the notation as $\mu^a\in\E^a[w]$ always means that $\mu^a$ is a section (and not an element) of $\E^a[w]$, global or local according to the context.

Conformal structure on  $M$ can be described by the \emph{conformal metric} $\mathbf{g}_{ab}$ which is a global section  of $\E_{(ab)}[2]$.
Any raising and lowering of indices is provided by the conformal metric, e.g. for $\mu^a\in\E^a[w]$ we have $\mu_a=\mathbf{g}_{ab}\mu^b\in\E_a[w+2]$.
A \textit{conformal scale} is an everywhere positive section of $\E[1]$.
The choice of scale $\si\in\E[1]$ corresponds to the choice of metric $g_{ab}\in\E_{(ab)}$ from the conformal class so that $g_{ab}=\si^{-2}\mathbf{g}_{ab}$.
The corresponding Levi-Civita connection is denoted as $\na$.
The Schouten tensor, which is a trace modification of the Ricci tensor, is denoted as $\Rho_{ab}$.
Transformations of quantities under the change of scale will be denoted by hats.
In particular, for $\wh\si=f\si$, $\Up_a=f^{-1}\na_a f$ and any $\mu^a\in\E^a$,
the Levi-Civita connection and the Schouten tensor change as
\begin{align}
\wh\na_a \mu^b &=\na_a \mu^b + \Up_a \mu^b - \mu_a\Up^b + \mu^c\Up_c\de_a{}^b, 
\label{transU} \\
\wh\Rho_{ab} &= \Rho_{ab} -\na_a\Up_b +\Up_a\Up_b -\tfrac12\Up_c\Up^c g_{ab}.
\label{transRho} 
\end{align}

\smallskip
Let $\Ga\subset M$ be a curve with a parametrization $t:\Ga\to\R$ and the tangent vector $U^a$ so that $U^a\nabla_a t=1$.
Regular curve is called space-like, time-like, respectively light-like (or null), in accord with the type of its tangent vector $U^a$, i.e. if $U_cU^c>0$, $U_cU^c<0$, respectively $U_cU^c=0$.
Of course, the curve may change the  type of its tangent vectors.
Since the whole study that follows is of very local nature, we restrict ourselves only to the (segments of) curves of fixed type.
In the first two cases, the density
\begin{equation}
u:=\sqrt{|U_cU^c|} \in \E[1]
\label{uu}
\end{equation}
is nowhere vanishing and will be employed later. Henceforth we always assume $\Ga$ is smooth.
Along the curve we use notation $\ddt{}:=U^c\nabla_c$ and 
\begin{equation}
U'^a:=\ddt{}U^a,\qquad
U''^a:=\ddt{2}U^a,\qquad
\dots,\qquad
\dif{U}{i}^a:=\ddt{i}U^a
\label{U^a}
\end{equation}
for the derived vectors. 
By abuse of notation we often write $U^a\in\E^a$ which should be read as $U^a\in\E^a|_\Ga$, 
and similarly for any other quantities defined only along the curve $\Ga$.
Note that the vector $\dif{U}{i}^a$ has order $i+1$ (with respect to $\Ga$).
Obviously, none of vectors in \eqref{U^a} is conformally invariant.
For instance, the acceleration vector transforms according to \eqref{transU} as
\begin{equation}
\wh{U}'{}^a=U'^a - U^cU_c\,\Up^a + 2U^c\Up_c\,U^a
\label{transacc}
\end{equation}
From this it follows that, for space- and time-like curves (but not for null curves), a metric in the conformal class may be chosen so that $U'^a=0$, i.e. the curve is an affinely parametrized geodesic of the corresponding Levi-Civita connection.
This indicates the problem with a conformally invariant notion of osculating subspaces.
Instead, we are going to make use of tractors.

\smallskip

The \emph{conformal standard tractor bundle} $\T$ over the conformal manifold $M$ of signature $(p,q)$, where $p+q=n=\dim M$, 
is the tractor bundle corresponding to the standard representation $\R^{p+1,q+1}$ of the conformal principal group 
$O(p+1,q+1)$. Specifically, $\T$ has rank $n+2$ and 
for any choice of scale, it is identified with the direct sum 
\begin{equation*}
\T=\E[1]\oplus\E^a[-1]\oplus\E[-1]
\end{equation*}
whose components change under the conformal rescaling as 
\begin{equation}
\pmat{\wh\si \\ \wh\mu^a \\ \wh\rho}=\pmat{\si \\ \mu^a+\Up^a\si \\ \rho-\Up_c\mu^c-\frac12\Up_c\Up^c\,\si}.
\label{transtr}
\end{equation}
Here $\Up_a$ is the 1-form corresponding to the change of scale as before.
Note that the projecting (or primary) slot is the top one.
The bundle $\T$ is endowed with  the \emph{standard tractor connection} $\tnabla$, the linear connection that is given by
\begin{equation*}
\tnabla_a \pmat{\si \\ \mu^b \\ \rho} = \pmat{\na_a\si-\mu_a \\ \na_a\mu^b+\de_a{}^b\rho+\Rho_a{}^b\si \\ \na_a\rho-\Rho_{ac}\mu^c}.
\end{equation*}
It follows this definition is indeed conformally invariant.
The bundle $\T$ also carries the \emph{standard tractor metric}, the bundle metric of signature $(p+1,q+1)$ that is schematically represented as
\begin{equation*}
\pmat{0&0&1\\0&\mathbf{g}_{ab}&0\\1&0&0},
\end{equation*}
i.e., for any sections $\form U=(\si,\mu^a,\rho)$ and $\form V=(\tau,\nu^a,\pi)$ of $\T$, 
\begin{equation*}
\g{\form U}{\form V}=\mu_a\nu^a+\si\pi+\rho\tau.
\end{equation*}
It follows the standard tractor metric is parallel with respect to the standard tractor connection.

\section{Relative conformal invariants}\label{Relative}
Throughout this section we consider conformal structures of positive definite signature.
The discussion for space- and time-like curves in indefinite signature shows only minor differences, the null case is more involved, see section \ref{Indefinite}.
In the first two subsections we only slightly extend the setting of \cite[section~2.8]{Bailey1994}.
Then we introduce a natural family of relative conformal invariants and the notion of conformal arc-length parameter of curve.

\subsection{Canonical lift and initial relations}\label{lift}
In the positive definite signature, any vector is space-like.
Hence, for a regular smooth curve $\Ga\subset M$ with a fixed parametrization $t$, 
the density \eqref{uu} becomes $u=\sqrt{U_cU^c}$ and it is nowhere vanishing.
This provides a lift of $\Ga$ to the standard tractor bundle $\T$ that is given by 
\begin{equation}
\form T:=\pmat{0\\0\\u^{-1}}.
\label{T}
\end{equation}
Along the curve we use the notation $\ddt{}:=U^c\tnabla_c$ and 
\begin{equation}
\form U:=\ddt{}\form T,\qquad
\dform U:=\ddt{2}\form T,\qquad
\dots,\qquad
\diform{U}{i}:=\ddt{i+1}\form T
\label{UA}
\end{equation}
for the derived tractors.
The notation is chosen so that the highest order term in the middle slot of $\diform{U}{i}$ is a multiple of $\dif{U}{i}^a$.
Note that the tractor $\diform{U}{i}$ has order $i+2$ (with respect to $\Ga$).
By construction, both the tractor lift $\form T$ and all tractors in \eqref{UA} are conformally invariant objects.
Explicitly, the first two derived tractors are
\begin{align}
\form U&=\pmat{0 \\ u^{-1}U^a \\ -u^{-3}U_cU'^c}, \label{U} \\
\dform U&=\pmat{-u \\ u^{-1}U'^a-2u^{-3} U_cU'^c\,U^a \\ -u^{-3}U_cU''^c-u^{-3}U'_cU'^c+3u^{-5}(U_cU'^c)^2-u^{-1}\Rho_{cd}U^cU^d}. \label{A}
\end{align}

The tractors $\form T,\form U,\dform U$ are linearly independent, and satisfy
\begin{equation}
\begin{matrix}
\g{\form T}{\form T}=0,&\qquad\g{\form T}{\form U}=0,&\qquad\g{\form T}{\dform U}=-1,\\
&\qquad\g{\form U}{\form U}=1,&\qquad\!\g{\form U}{\dform U}=0.\hfill
\end{matrix}
\label{starting}
\end{equation}
The first nontrivial identity is  
\begin{equation}
\g{\dform U}{\dform U} = 3u^{-2}U'_cU'^c +2u^{-2}U_cU''^c -6u^{-4}(U_cU'^c)^2 +2\Rho_{cd}U^cU^d.
\label{AA}
\end{equation}
In order to simplify expressions, we will use the following notation 
\begin{equation*}
\al:=\g{\dform U}{\dform U},\qquad
\be:=\g{\ddform U}{\ddform U},\qquad
\ga:=\g{\dddform U}{\dddform U},\qquad
\text{etc.}
\end{equation*}
Initial relations above and their consequences may be schematically indicated by the Gram matrix which, for the sequence 
$(\form T,\form U,\dform U,\ddform U,\dddform U)$, has the form
\begin{equation}
\pmat{
0&0&-1&0&\al\\[0.5ex]
0&1&0&-\al&-\tfrac32 \al'\\[0.5ex]
-1&0&\al&\tfrac12 \al'&\tfrac12 \al''-\be\\[0.5ex]
0&-\al&\tfrac12 \al'&\be&\tfrac12 \be'\\[0.5ex]
\al&-\tfrac32 \al'&\tfrac12 \al''-\be&\tfrac12 \be'& \ga}.
\label{starting-gram}
\end{equation}
Clearly, the generating rule may be written as
\begin{align*}
\g{\diform{U}{i}}{\diform{U}{i+j+1}}=\big(\g{\diform{U}{i}}{\diform{U}{i+j}}\big)'-\g{\diform{U}{i+1}}{\diform{U}{i+j}}.
\end{align*}

Later we will need some details on the next derived tractor $\ddform U$.
As a consequence of $\g{\form T}{\ddform U}=0$, the projecting slot of $\ddform U$ must vanish.
With the help of \eqref{AA}, the middle slot of $\ddform U$ may be expressed as 
\begin{equation*}
u^{-1}U''^a -3u^{-3}U_cU'^c\,U'^a +\left(\tfrac32u^{-3}U'_cU'^c -\tfrac32u^{-1}(\g{\dform U}{\dform U}) +2u^{-1}U^cU^d\Rho_{cd}\right)U^a -uU^c\Rho_c{}^a.
\end{equation*}

\subsection{Reparametrizations} \label{repa}
Let $\ti t=g(t)$ be a reparametrization of the curve $\Ga$.
All objects related to the new parameter $\ti t$ will be denoted by tildes in accord with $\dd{\ti t}{}=g'^{-1}\ddt{}$, where $g'=\frac{dg}{dt}$.
In particular, $\wt U^a=g'^{-1} U^a$, $\wt u=g'^{-1}u$ and $\tform T=g'\form T$. 
Using just the chain rule and the Leibniz rule, one easily verifies that 
\begin{align}
\begin{split}
\tform U&=\form U+g'^{-1}{g''}\form T, \\
\tdform U&=g'^{-1}\dform U+g'^{-2}g''\form U+(g'^{-2}g'''-g'^{-3}g''^2)\form T, \\
\tddform U&=g'^{-2}\ddform U+2g'^{-2}\cS(g)\form U+g'^{-2}\cS(g)'\ \form T,
\label{Utilde}
\end{split}
\end{align}
where $\cS$ is the \emph{Schwarzian derivative},
\begin{equation*}
\cS(g):=\frac{g'''}{g'}-\frac32\left(\frac{g''}{g'}\right)^2
=\left( \ln g' \right)''-\frac12\left( \ln g' \right)'^2.
\end{equation*}

According to the starting relations \eqref{starting}, we have
\begin{align}
\begin{split}
\g{\tdform U}{\tdform U} &=g'^{-2}\left(\g{\dform U}{\dform U} -2 \cS(g)\right), \\
\g{\tddform U}{\tddform U} &=g'^{-4}\left(\g{\ddform U}{\ddform U} -4\cS(g)\g{\dform U}{\dform U} +4\cS(g)^2\right).
\end{split}
\label{tiAA}
\end{align}
Hence vanishing of $\g{\dform U}{\dform U}$ determines a natural projective structure on any curve, cf. \cite[Proposition 2.11]{Bailey1994}:

\begin{proposition}[\cite{Bailey1994}]\label{prop-repa}
The equation $\g{\dform U}{\dform U}=0$, regarded as a condition on the parametrization of a curve, 
determines a preferred family of parametrizations with freedom given by the projective group of the line.
\end{proposition}
Any parameter from this family will be called \emph{projective}.

From \eqref{tiAA} it further follows that 
\begin{align*}
\g{\tddform U}{\tddform U}-(\g{\tdform U}{\tdform U})^2=g'^{-4}\left(\g{\ddform U}{\ddform U}-(\g{\dform U}{\dform U})^2\right).
\end{align*}
Hence the function 
\begin{equation}
\Ph:=\g{\ddform U}{\ddform U}-(\g{\dform U}{\dform U})^2 =\be-\al^2
\label{th4}
\end{equation}
is a relative conformal invariant of the curve.

\subsection{Relative conformal invariants} \label{relative}
In general, a \emph{relative conformal invariant of weight $k$} of the curve $\Ga$ is 
a conformally invariant function $I:\Ga\to\R$ that transforms under a reparametrization $\ti t=g(t)$ of the curve as 
\begin{equation*}
\wt I=g'^{-k}I.
\end{equation*}
In particular,  vanishing, respectively nonvanishing, of any relative invariant is independent on reparametrizations.
Conformal invariants of weight 0 are the \textit{absolute} invariants.

Let us denote 
\begin{align}
\De_i:=\det\left(\operatorname{Gram}\left(\form T,\form U,\dform U,\dots,\diform{U}{i-2}\right)\right),
\label{gram}
\end{align}
the determinant of the Gram matrix corresponding to the first $i$ tractors from the derived sequence $\form T,\form U,\dform U,\dots$.
From \eqref{starting} it is obvious that $\De_1=\De_2=0$ and $\De_3=-1$ independently of the curve and its parametrization.
These three determinants are thus absolute, but trivial conformal invariants.
The first nontrivial invariant is $\De_4$.
From \eqref{starting-gram} it follows that 
\begin{align}
\De_4=-\be+\al^2=-\Ph.
\label{De4}
\end{align}
In general, we have

\begin{lemma}\label{lem-gram}
For $i=4,\dots,n+2$, the Gram determinant $\De_i$ is a relative conformal invariant of weight $i(i-3)$.
\end{lemma}

\begin{proof}
As a generalization of \eqref{Utilde} we have 
\begin{align*}
\tdiform{U}{j}= g'^{-j}\diform{U}{j} \mod\spanb{\form T,\form U,\dform U,\dots,\diform{U}{j-1}},
\end{align*}
for any $j=1,\dots,n$.
From this and properties of the determinant it follows that \eqref{gram} changes under the reparametrization as
\begin{align*}
\wt\De_i=g'^{-2\cdot2}\cdots g'^{-2(i-2)} \De_i =g'^{-i(i-3)} \De_i,
\end{align*}
for any $i=4,\dots,n+2$.
\end{proof}

For $i=4,\dots,n+2$, the tractor metric restricted to $\spanb{\form T,\form U,\dots,\diform{U}{i-2}}$ is nondegenerate, thus  vanishing of $\De_i$ is equivalent to the fact that the determining tractors are linearly dependent.
In particular, for $i\ge4$, vanishing of $\De_i$ implies vanishing of $\De_{i+1}$.
Note also that the weight of any $\De_i$ is even, thus its sign is independent of all reparametrizations.

\subsection{Conformal arc-length}\label{arc}
By \eqref{th4}, respectively \eqref{De4}, we defined a nontrivial relative conformal invariant of the lowest possible weight and order.
It is actually nonnegative:

\begin{lemma}\label{Ph}
$\Ph\ge0$.
\end{lemma}
\begin{proof}
For a parametrization belonging to the projective family of Proposition \ref{prop-repa} we have $\Ph=\g{\ddform U}{\ddform U}$.
Since the top slot of $\ddform U$ vanishes, the previous expression equals to the norm squared of the middle slot of $\ddform U$ and so it is nonnegative.
Since the weight of $\Ph$ is even, it is a nonnegative function independently of the parametrization of the curve.
\end{proof}

Suppose that $\Ga$ is a curve with nowhere vanishing $\Ph$.
Then 
\begin{equation*}
ds:=\sqrt[4]{\Ph(t)}\,dt
\end{equation*}
is a well-defined conformally invariant 1-form along the curve whose integration yields a distinguished parametrization of the curve; it is given uniquely up to an additive constant.
Such parameter is called the \emph{conformal arc-length}.
Note that if $s$ is the conformal arc-length then $\Ph(s)=1$.

It will be clear from later alternative expressions that this is the same distinguished parameter as one finds in literature, e.g. in \cite{Fialkow1942} or \cite{Cairns1994}.
In the later reference, one also finds a notion of \textit{vertex}, which is the point of curve where $\Ph=0$.
The vertices of curves are clearly invariant under conformal transformations.
Generic curves are vertex-free, the opposite extreme is discussed in the next section.

\begin{remark}\label{indef2}
For space- and time-like curves in the case of general indefinite signature, 
the relative conformal invariant $\Ph$ defined by \eqref{th4}, respectively \eqref{De4}, is not necessarily nonnegative as in Lemma \ref{Ph}.
It may even happen that it vanishes although the corresponding tractors $\form T,\form U,\dform U,\ddform U$ are linearly independent.
The notion of vertex in such cases would rather be defined by the linear dependence of these tractors than by vanishing of $\Ph$, cf. Remark \ref{indef3}(2).
\end{remark}

\section{Conformal circles and conserved quantities}\label{Geod}
We continue the exposition with the assumption of positive definite signature.
The only difference in indefinite signature concerns the notion of conformal circles that are defined just for the space- or time-like directions.
Reasonable analogies for null directions are discussed in subsection \ref{null-geod}.

\subsection{Conformal circles}\label{geod}
In this subsection we consider the curves for which $\Ph$ vanishes identically, i.e. the curves consisting only of vertices.
The conformal arc-length is not defined for such curves and they are excluded from the discussion in the main part of this paper as the most degenerate cases.
However, these curves form a very distinguished family of curves that coincide with the so-called \textit{conformal circles}, the curves which are in some sense the closest conformal analogues of usual geodesics on Riemannian manifolds.
In the flat case, they are just the circles, respectively straight lines, i.e. the conformal images of Euclidean geodesics.

Conformal circles are well known and studied
(under various nicknames%
\footnote{%
E.g. they are called conformal null curves in \cite{Fialkow1942} or conformal geodesics in  \cite{Friedrich1987} and \cite{Herzlich2012}.
Note that the conformal geodesics studied in \cite{Fialkow1939} or \cite{Musso1994} form a different, more general, class of curves.%
})
in the literature.
A quick survey with an explanation of the relationship of the definition below to distinguished curves of Cartan's conformal connection can be found in \cite[section~5.1]{Herzlich2012}.
Conformal circles on a general conformal manifold can also be related to ordinary circles in the flat model via the notion of Cartan's development of curves.

Here we adapt the notation of \cite{Bailey1990}, where \textit{conformal circles} are defined as the solutions to the system of conformally invariant third order ODE's 
\begin{equation}
U''^a = 3u^{-2}U_cU'^c\,U'^a-\tfrac32u^{-2}U'_cU'^c\,U^a+u^2U^c\Rho_c{}^a-2U^cU^d\Rho_{cd}\,U^a.
\label{BE5}
\end{equation}
Contracting, respectively skew symmetrizing \eqref{BE5}, with $U^a$ yields the following pair of equations equivalent to \eqref{BE5}:
\begin{align}
U_aU''^a &= 3u^{-2}(U_cU'^c)^2-\tfrac32U'_cU'^c-u^2 U^cU^d\Rho_{cd},
\label{BE6} \\
U_{[a}U''_{b]} &= 3u^{-2}U_cU'^c\,U_{[a}U'_{b]}+u^2U^c\Rho_{c[b}U_{a]}.
\label{BE7}
\end{align}
From \eqref{transacc} we know that, for any curve with an arbitrary parametrization, a metric in the conformal class may be chosen so that it is an affinely parametrized geodesic of the corresponding Levi-Civita connection, i.e. $U'^a=0$.
It is then easy to see from \eqref{BE6} and \eqref{BE7} that 

\begin{proposition}[\cite{Bailey1990}]\label{prop-ge}
The curve is a conformal circle if and only if there is a metric in the conformal class such that the curve is an affinely parametrized geodesic and $U^c\Rho_{ca}=0$.
\end{proposition}

From the explicit expressions in subsection \ref{lift} we may read the following tractorial interpretations of the  equations above:
The equation \eqref{BE5} is equivalent to vanishing of the middle slot of the tractor 
$\ddform U+\tfrac32(\g{\dform U}{\dform U})\,\form U$
(since the primary slot is zero, this is indeed a conformally invariant condition).
This implies that the middle slots of tractors $\ddform U$ and $\form U$ are collinear, which is just the condition \eqref{BE7}.
In particular, this condition is independent of any reparametrization, i.e.
solutions to \eqref{BE7} are the {conformal circles} parametrized by an arbitrary parameter.
The equation \eqref{BE6} is equivalent to $\g{\dform U}{\dform U}=0$. 
In particular, solutions to \eqref{BE6} correspond to the projective parameters of the curve.
Considering further the injecting slot of the tractor $\ddform U$, it is shown in \cite[Proposition 2.12]{Bailey1994} that 

\begin{proposition}[\cite{Bailey1994}]\label{prop-geo}
The curve is a projectively parametrized conformal circle if and only if it obeys
\begin{equation}
\g{\dform U}{\dform U}=0 \qquad\text{and}\qquad \ddform U=0.
\label{AAB}
\end{equation}
\end{proposition}

We may easily generalize this statement as follows:

\begin{proposition}\label{prop-geod}
The following conditions are equivalent:
\begin{enumerate}[(a)]
\item the curve is a conformal circle (with and arbitrary parametrization),
\item all its points are vertices, i.e. 
\begin{equation*}
\Ph=\g{\ddform U}{\ddform U}-(\g{\dform U}{\dform U})^2=0,
\end{equation*}
\item the rank 3 subbundle $\spanb{\form T,\form U,\dform U}\subset\T$ is parallel along the curve.
\end{enumerate}
\end{proposition}

\begin{proof}
From subsection \ref{lift} we know that the tractors  $\form T,\form U,\dform U$ are linearly independent for any curve and its parametrization.
From subsection \ref{relative} we know that $\Ph=0$ is equivalent to the fact that the tractors $\form T,\form U,\dform U,\ddform U$ are linearly dependent, i.e. $\ddform U$ belongs to the subbundle $\spanb{\form T,\form U,\dform U}\subset\T$ which is therefore parallel. 
Thus (b) and (c) are equivalent.

Among tractors $\form T,\form U,\dform U,\ddform U$, only $\dform U$ has nonzero projecting slot.
Hence the previous condition means that $\ddform U$ is a linear combination of $\form T$ and $\form U$.
Since the middle slot of $\form T$ vanishes, the previous is equivalent to the middle slots of $\ddform U$ and $\form U$ being collinear, which is just the condition \eqref{BE7}.
Thus (c) and (a) are equivalent.
\end{proof}

\begin{remarks}\label{indef3}
(1)
Note the two conditions in \eqref{AAB} are not independent: 
if $\ddform U=0$ then $\g{\dform U}{\dform U}=0$.

(2)
Conformal circles in the case of indefinite signature exist only for space- or time-like directions.
They have all of the above stated properties except the condition (b) of Proposition \ref{prop-geod}, 
which is no more equivalent to the others, cf.\ Remark \ref{indef2}.
\end{remarks}

\subsection{Conserved quantities}\label{conserve}
As an application of the current approach we describe some conserved quantities (or first integrals) along conformal circles for some specific conformal structures. Here we consider conformal structures admitting almost Einstein scales or conformal 
Killing fields with some additional  property.
The description of the conserved quantities is easily given due to the tractorial characterization of these additional data, which can be found e.g. in \cite{GoEinst}.

An \textit{almost Einstein scale} is a section $\si\in\E[1]$ satisfying the conformally invariant condition that the trace-free part of 
$\na_a\na_b\si+\Rho_{ab}\si\in\E_{ab}[1]$ vanishes.  
Then $\si$ is nonvanishing on an open dense set where the metric $g_{ab}=\si^{-2}\mathbf{g}_{ab}$ is Einstein.
The tractorial characterization of this fact is  that the tractor field $L^{\T}(\si)\in\T$ is parallel, where 
$L^{\T}:\E[1]\to\T$ denotes the BGG splitting operator. 

\begin{proposition}\label{prop-ein}
Let $\Ga$ be a conformal circle with a projective parameter $t$. 
Assume the conformal manifold $M$ admits an almost Einstein scale $\si\in\E[1]$ and let $\form S=L^{\T}(\si)\in\T$ be the corresponding parallel tractor.
Then the function $\frak s:=  \g{\dform U}{\form S}$ is a conserved quantity along $\Ga$, i.e.\ $\frac{d}{dt} \frak s=0$.
\end{proposition}

\begin{proof}
A projectively parametrized circle satisfies $\ddform U=0$, 
the tractor $\form S$ corresponding  to an almost Einstein scale satisfies $\dform S=0$.
Thus $\ddt{}(\g{\dform U}{\form S})=0$.
\end{proof}

\smallskip
A \textit{conformal Killing field} $k^a\in\E^a$ is an infinitesimal symmetry of the conformal structure.
This is equivalent to the vanishing of the trace-free part of $\na_{(a}k_{b)}\in\E_{(ab)}[2]$.
The tractorial characterization of this fact is  that the tractor field $\form K:=L^{\La^2\T}(k^a)\in\La^2\T$ satisfies 
\begin{align}
\tnabla_a\form K = k^c\bs\Om_{ca}, 
\label{Omk}
\end{align}
where $L^{\La^2\T}:\E^a\to\La^2\T$ denotes the BGG splitting operator and $\bs\Om_{cd}\in\E_{[cd]}\otimes\La^2\T$ is the curvature of the tractor connection $\tnabla$. 
Further, along the given curve, 
we introduce the 1-form 
\begin{align}
\ell_a := U^b \bs\Om_{ab}(\form U, \form U').
\label{ell}
\end{align}

\begin{proposition}\label{prop-ckf}
Let $\Ga$ be a conformal circle with a projective parameter $t$. 
Assume the conformal manifold $M$ admits a conformal Killing field $k^a\in\E^a$ satisfying $k^a\ell_a=0$, and let $\form K:=L^{\La^2\T}(k^a)\in\La^2\T$ be the corresponding tractor field.
Then the function $\frak k := \form K(\form U,\dform U)$ is a conserved quantity along $\Ga$, i.e.\ $\frac{d}{dt} \frak k=0$.
\end{proposition}

\begin{proof}
A projectively parametrized circle satisfies $\ddform U=0$. 
The tractor $\form K$ corresponding  to $k^a$ satisfies \eqref{Omk} and hence 
$(\frac{d}{dt} \form K)(\form U,\dform U) = k^a\ell_a =0$.
Also, $\form K$ is skew, thus $\ddt{}(\form K(\form U,\dform U)) = 0$.
\end{proof}

Note that the assumption $k^a\ell_a=0$ in the Proposition is automatically satisfied for \textit{normal}  conformal Killing fields which are characterized by the vanishing of \eqref{Omk}.

\begin{remarks}
(1)
It is somewhat tedious to expand tractorial formulas for the conserved quantities $\frak s$ and $\frak k$.
To proceed, one needs the respective splitting operators and some manipulation.
From the explicit expressions it in particular follows that the considered quantities are nontrivial.
All that can be found already in \cite{Snell2015}.
Note also that both $\frak s$ and $\frak k$ are quantities of order 2, while the conformal circle equation is of order 3.

(2)
Expanding \eqref{ell} yields
\begin{align*}
\ell_a = U^b\sfW_{abcd} U^cU'{}^d - 2u^2 U^b\na_{[a} \Rho_{b]c}U^c,
\end{align*}
where $\sfW_{ab}{}^c{}_d$ is the conformal Weyl tensor.
Remarkably, the condition $\ell_a=0$ plays a role in the variational approach to conformal circles, see \cite[p.~218]{Bailey1990}.
\end{remarks}

\section{Absolute conformal invariants}\label{Absolute}

For a generic curve, the relative conformal invariants $\De_i$ from section \ref{Relative} are all nontrivial, provided that $4\le i\le n+2$.
One may therefore easily build a galaxy of absolute conformal invariants.
In this section we construct a  minimal set of such invariants, which is done in a natural if not canonical way.
In order to make the construction complete, the assumption of positive definite signature is important here.
In particular, the standard tractor metric has signature $(n+1,1)$.

\subsection{Tractor Frenet formulas and conformal curvatures}\label{frenet}
Let $\Ga $ be a generic curve, let $s$ be its conformal arc-length parameter and let $\form T, \form U=\dds{}\form T, \dform U=\dds{}\form U, \dots$ be the corresponding tractors as above.
The restriction of the tractor metric to the subbundle $\spanb{\form T,\form U,\dform U}$ is nondegenerate and has signature $(2,1)$.
Thus its orthogonal subbundle in $\T$ is complementary and has  positive definite signature.
We are going to transform the initial tractor frame $(\form T, \form U, \dform U, \dots)$ into a natural pseudo-orthonormal one.
Firstly, under the transformation
\begin{equation}
\form U_0:=\form T,\qquad
\form U_1:=\form U,\qquad
\form U_{2}:=-\dform U-\tfrac12(\g{\dform U}{\dform U})\,\form T,
\label{U0}
\end{equation}
the initial relations \eqref{starting} transform to
\begin{equation}
\begin{matrix}
\g{\form U_0}{\form U_0}=0,&\qquad\g{\form U_0}{\form U_1}=0,&\qquad\hfill\g{\form U_0}{\form U_{2}}=1,\\
&\qquad\g{\form U_1}{\form U_1}=1,&\hfill\g{\form U_1}{\form U_{2}}=0,\\
&&\qquad\g{\form U_{2}}{\form U_{2}}=0.
\end{matrix}
\label{newrel}
\end{equation}
Secondly, by the standard orthonormalization process we may complete this triple to a tractor frame 
\begin{equation}
(\form U_0, \form U_1, \form U_2; \form U_3, \dots,\form U_{n+1})
\label{trfr}
\end{equation}
such that $\form U_i$ belongs to the span $\spanb{\form T,\form U,\dots,\diform{U}{i-1}}$, for all admissible $i$, and the corresponding Gram matrix is
\begin{equation}
\mbox{\footnotesize{$
\left(
\begin{BMAT}(@)[3pt]{ccc0ccc}{ccc0ccc}
&&1&&& \\ 
&1&&&& \\ 
1&&&&& \\
&&&1&& \\
&&&&\ddots& \\
&&&&&1 \\
\end{BMAT}
\right).
$}}
\label{grU}
\end{equation}
More concretely, for any $i=3,\dots,n$, the frame tractors are
\begin{equation}
\begin{gathered}
\form U_i=\form V_i/\sqrt{\g{\form V_i}{\form V_i}},
\qquad \text{where} \\
\form V_i=\diform{U}{i-1}-(\g{\diform{U}{i-1}}{\form U_{i-1}})\,\form U_{i-1}-\cdots 
\hskip3cm \\ \hskip1cm
\cdots-(\g{\diform{U}{i-1}}{\form U_0})\,\form U_{2}-(\g{\diform{U}{i-1}}{\form U_1})\,\form U_1-(\g{\diform{U}{i-1}}{\form U_{2}})\,\form U_0. 
\label{Ui}
\end{gathered}
\end{equation}
The last tractor $\form U_{n+1}$ is determined by the orthocomplement in $\T$ to the codimension one subbundle 
$\spanb{\form T,\dots,\diform{U}{n-1}}=\spanb{\form U_0,\dots,\form U_n}$ 
up to orientation.
We do not need to fix the orientation now, it is done only later.
The frame \eqref{trfr} is called the \emph{tractor Frenet frame} associated to the curve~$\Ga$.

\smallskip
Now, differentiating the constituents of the tractor Frenet frame with respect to the conformal arc-length $s$ and expressing the result within that frame yields Frenet-like identities.
As in the classical situation, 
it follows that the coefficients of that system determine the generating set of absolute invariants of the curve.
The pseudo-orthonormality of the tractor Frenet frame implies a lot of symmetries among these coefficients.

Just from \eqref{U0} we have
\begin{align}
\dform U_0&=\form U_1, \label{U0'} \\
\dform U_1&=K_1\form U_0-\form U_{2}, \label{U1'}
\qquad\text{where}\qquad
K_1:=-\tfrac12\g{\dform U}{\dform U}
\end{align}
is the first nontrivial coefficient, the \textit{first conformal curvature} of the curve.
In general, for any admissible $i$, the derivative $\dform U_i$ is a linear combination of tractors $\form U_0,\form U_1,\dots,\form U_{i+1}$, namely,
\begin{align*}
\dform U_i=(\g{\dform U_i}{\form U_2})\,\form U_0+(\g{\dform U_i}{\form U_1})\,\form U_1+(\g{\dform U_i}{\form U_0})\,\form U_2+(\g{\dform U_i}{\form U_3})\,\form U_3+\dots+(\g{\dform U_i}{\form U_{i+1}})\,\form U_{i+1},
\end{align*}
where the coefficients satisfy $\g{\dform U_i}{\form U_j}=-\g{\form U_i}{\dform U_j}$ and they have to vanish if $i=j$ or $j\ge i+2$.
In particular, from \eqref{U0'} we may read $\g{\dform U_0}{\form U_1}=1$ and $\g{\dform U_0}{\form U_j}=0$, for $j=0,2,\dots,n+1$.
Similarly, from \eqref{U1'} we further read $\g{\dform U_1}{\form U_2}=K_1$ and $\g{\dform U_1}{\form U_j}=0$, for $j=1,3,\dots,n+1$.
Substituting these facts into the previous display, for $i=2$, we see that the only unknown coefficient is the one of $\form U_3$, i.e. $\g{\dform U_2}{\form U_3}$.
It however follows that this coefficient is constant so that we have

\begin{lemma}\label{lem-frenet}
\begin{align}
\dform U_2=-K_1\form U_1-{\form U_3}.
\label{U2'}
\end{align}
\end{lemma}

\begin{proof}
From \eqref{Ui}, according to \eqref{U0} and \eqref{starting}, we obtain
\begin{align*}
\form V_3
=\ddform U-(\g{\ddform U}{\form U})\,\form U+(\g{\ddform U}{\dform U})\,\form T.
\end{align*}
Using $\g{\ddform U}{\form U}=-\g{\dform U}{\dform U}$, which is  a consequence of $\g{\form U}{\dform U}=0$, we obtain
\begin{equation*}
\g{\form V_3}{\form V_3}=\g{\ddform U}{\ddform U}-(\g{\dform U}{\dform U})^2=\Ph=1
\end{equation*}
and thus  $\form U_3=\form V_3$. 
From \eqref{U0} and \eqref{Ui} we know that $\dform U_2$ equals to $-\ddform U \mod\spanb{\form T,\form U}$, respectively $-\form V_3\mod\spanb{\form U_0,\form U_1}$.
Hence $\g{\dform U_2}{\form U_3}=-\g{\form V_3}{\form U_3}=-1$.
\end{proof}

Following the same ideas as before, one gradually and easily obtains
\begin{align}
  \begin{split}
    &\dform U_3=\form U_0+K_2\form U_4,\\
    &\dform U_i=-K_{i-2}\form U_{i-1}+K_{i-1}\form U_{i+1},\qquad\text{for}\ i=4,\dots,n,\\
    &\dform U_{n+1}=-K_{n-1}\form U_{n},
    \label{trfrfor}
  \end{split}
\end{align}
where $K_2,\dots,K_{n-1}$ are the \emph{higher conformal curvatures} of the curve.
The bunch of equations \eqref{U0'}, \eqref{U1'}, \eqref{U2'} and \eqref{trfrfor} form the \textit{tractor Frenet equations} associated to the curve.
Schematically they may be written as
\begin{equation}
\mbox{\footnotesize{$
\left(
\begin{BMAT}(@)[1.2pt]{c}{ccc0cccc}
\dform U_0\\ \dform U_1\\ \dform U_2\\ \dform U_3\\ \dform U_4\\ \vdots\\  \dform U_{n+1}
\end{BMAT}
\right)
=
\left(
\begin{BMAT}(@)[1.5pt]{ccc0cccc}{ccc0ccccc}
&1&&&&& \\
K_1&&-1&&&& \\
&-K_1&&-1&&& \\
1&&&&\ \ K_2&& \\
&&&-K_2&&& \\
&&&&\ddots&\ \ \ \ddots&\\ 
&&&&&&K_{n-1} \\
&&&&&-K_{n-1}& \\
\end{BMAT}
\right)
\cdot
\left(
\begin{BMAT}(@)[1.9pt]{c}{ccc0cccc}
\form U_0\\ \form U_1\\ \form U_2\\ \form U_3\\ \form U_4\\ \vdots\\ \form U_{n+1}
\end{BMAT}
\right)
$}}.
\label{trfrform}
\end{equation}

Altogether, we summarize as follows:

\begin{proposition}
The system of tractor Frenet equations determines a generating set of absolute conformal invariants of the curve.
On a conformal Riemannian manifold of dimension $n$, it consists of $n-1$ conformal curvatures that are expressed,
with respect to the conformal arc-length parametrization, 
as
\begin{align}
K_i=
\begin{cases}
\g{\dform U_1}{\form U_{2}},&\qquad\text{for}\ i=1, \\
\g{\dform U_{i+1}}{\form U_{i+2}},&\qquad\text{for}\ i=2,\dots,n-1. 
\end{cases}
\label{Kie}
\end{align}
\end{proposition}

In \eqref{U1'} we have an alternative expression of the first conformal curvature in terms of initial tractors.
Thus, as a consequence of Proposition \ref{prop-repa} and the ensuing definition, we obtain an interpretation of that invariant:

\begin{proposition}\label{prop-K1}
The conformal arc-length parameter belongs to the projective family of parameters if and only if $K_1=0$.
\end{proposition}

\begin{remarks} \label{rem-frenet}
(1)
The matrix in \eqref{trfrform} is skew with respect to the inner product corresponding to \eqref{grU}.
A change of basis leads to a different matrix realization, 
e.g., just a permutation of the frame tractors leads to following rewrittening of \eqref{trfrform}:
\begin{equation*}
\mbox{\footnotesize{$
\left(
\begin{BMAT}(@)[1.2pt]{c}{c0ccccc0c}
\dform U_0\\ \dform U_1\\ \dform U_3\\ \vdots\\ \dform U_n\\ \dform U_{n+1}\\  \dform U_{2}
\end{BMAT}
\right)
=
\left(
\begin{BMAT}(@)[1.5pt]{c0cccccc0c}{c0cccccc0c}
&1&&&&&&\\ 
K_1&&&&&&&-1\\ 
1&&&\ \ K_2&&&&\\ 
&&-K_2&&&&&\\ 
&&&\ddots&&\ \ \ddots&&\\ 
&&&&&&K_{n-1}&\\
&&&&&-K_{n-1}&&\\
&-K_1&-1&&&&&
\end{BMAT}
\right)
\cdot
\left(
\begin{BMAT}(@)[1.9pt]{c}{c0ccccc0c}
\form U_0\\ \form U_1\\ \form U_3\\ \vdots\\ \form U_n\\ \form U_{n+1}\\  \form U_{2}
\end{BMAT}
\right)
$}}.
\end{equation*}
This block decomposition reflects the standard grading of the Lie algebra $\frak{so}(n+1,1)$ with the parabolic subalgebra corresponding to the block lower 
triangular matrices.
This choice is very close to the ones in \cite{Sulanke1981} or \cite{Cairns1994}.

(2)
For generic curves, the $(n+2)$-tuple of derived tractors is linearly independent.
For specific curves, this is not the case and 
we cannot build the full tractor Frenet frame and all tractor Frenet formulas.
One may however mimic the previous procedure as long as the tractors are independent:

If, say, the tractors $\form T, \form U, \dform U, \dots,\diform{U}{i}$ are linearly independent and $\diform{U}{i+1}$ belongs to their span, then they form a parallel subbundle in $\T$ of rank $i+2$ with the corresponding orthonormal tractors $\form U_0,\form U_1,\form U_2,\dots,\form U_{i+1}$.
This yields a part of the tractor Frenet formulas above, in which the first $i-1$ conformal curvatures occur; the remaining ones may be considered to be zero.
(In this vein, conformal circles may be considered as the most degenerate curves whose all conformal curvatures vanish, cf. Proposition \ref{prop-geod}.)
In the flat case, this means that the curve itself is contained in an $i$-sphere, respectively $i$-plane, i.e. in the  conformal image of an $i$-dimensional Euclidean subspace.
See also \cite[section 12]{Fialkow1942} for further details.

(3)\label{indef4}
For space- and time-like curves in the case of indefinite signature, we remark the following:
The same transformations as in \eqref{U0} lead to the relations \eqref{newrel} where the type of curve is reflected just in the sign of the middle element, hence also in the signature of the  subbundle $\spanb{\form T,\form U,\dform U}$ in $\T$.
The orthogonal complement of this subbundle in $\T$ has in general indefinite signature.
Thus, the tractors $\form V_i$ from \eqref{Ui} may have any sign, including zero. 
If the sign is negative, we just adjust the corresponding definition of $\form U_i$
and, accordingly, one sign in the bottom-right block of \eqref{grU} changes.
If the tractor $\form V_i$ is isotropic then we cannot satisfy both the diagonal form of the bottom-right block of \eqref{grU} and the condition $\form U_i\in\spann{\form T,\form U,\dots,\diform{U}{i-1}}$.
Of course, one can always transform given tractors to a pseudo-orthonormal tractor frame, but only with some additional choices. 
The rest remains basically the same, only the tractor Frenet formulas may look different and some of the resulting invariants has to be related to the respective choices.
(Nevertheless, it is clear that the exceptional conformal curvature $K_1$ and its interpretation as in Proposition \ref{prop-K1} are not influenced by these issues.)
These observations leads to a finer type characterization of space- and time-like curves.
Its complete branching structure is manageable only in a concrete dimension and signature.
\end{remarks}

\subsection{Conformal curvatures revised}\label{revise}
In the definition above, the orientation of the last tractor from the tractor Frenet frame  was a matter of choice. 
In what follows we assume the orientation of $\form U_{n+1}$ is chosen so that it belongs to the same half-space as $\diform{U}{n}$ with respect to the hyperplane $\spanb{\form U_0,\dots,\form U_n}$.
Thus, according to  \eqref{Ui}, we have
\begin{align}
  \diform{U}{i} = \form V_{i+1} \mod \spanb{\form U_0,\form U_1,\dots,\form U_{i}}, 
\label{UiVj}
\end{align}
for any $i=3,\dots,n$.
Also, from definitions and the tractor Frenet formulas, one progressively obtains
\begin{gather}
\begin{split}
\form T=\form U_0, \qquad
\form U=\form U_1, \qquad
\dform U=-\form U_{2}+K_1\form U_0, \qquad
\ddform U=\form U_3+2K_1\form U_1+K_1'\form U_0, \qquad
\text{etc.}
\end{split}
\end{gather}
In particular, it holds 
$\ddform U = \form U_3 \mod \spanb{\form U_0,\form U_1,\form U_{2}}$.
Similarly, it follows that
$\dddform U = K_2\form U_4 \mod \spanb{\form U_0,\form U_1,\form U_2,\form U_3}$ and, inductively, 
\begin{align}
  \diform{U}{i} = K_2\cdots K_{i-1}\form U_{i+1} \mod \spanb{\form U_0,\form U_1,\dots,\form U_{i}}, 
\label{UiKj}
\end{align}
for $i=3,\dots,n$.
Altogether, from \eqref{UiVj}, \eqref{UiKj} and $\form U_i=\form V_i/\sqrt{\g{\form V_i}{\form V_i}}$, we obtain 
\begin{align*}
K_2\cdots K_{i-1} =\sqrt{\g{\form V_{i+1}}{\form V_{i+1}}}, 
\end{align*}
which leads to the following conclusion:

\begin{proposition}\label{prop-revise}
With respect to the conformal arc-length parametrization, 
the higher conformal curvatures can be expressed as 
\begin{align}
K_i =\sqrt{\frac{\g{\form V_{i+2}}{\form V_{i+2}}}{\g{\form V_{i+1}}{\form V_{i+1}}}},
\qquad\text{for $i=2,\dots,n-1$}.
\label{KiVj}
\end{align}
\end{proposition}

In particular, all these curvatures are positive.
The first conformal curvature $K_1$ does not fit to this uniform description and  its sign may be arbitrary, including zero.

An expansion of \eqref{KiVj}, according to \eqref{U0} and \eqref{Ui}, provides expressions of higher conformal curvatures in terms of initial tractors.
Instead of exploiting this demanding procedure, we are going to use the relative conformal invariants from subsection \ref{relative}.
From the construction of the tractor Frenet frame it follows that 
$\De_{i+1}=\De_i (\g{\form V_i}{\form V_i})$, for all admissible $i$.
Since $\De_3=-1$, we see that all admissible determinants are negative.
Hence we can express $\g{\form V_i}{\form V_i}$ as a quotient of determinants, whose substitution into \eqref{KiVj} yields
\begin{align}
K_i=\frac{\sqrt{\De_{i+1}\De_{i+3}}}{-\De_{i+2}}.
\label{KiDe}
\end{align}

\medskip

Up to this point, everything has been related to the parametrization by the conformal arc-length. 
Now we consider an arbitrary parametrization of the curve.

\begin{theorem} \label{husty-thm}
With respect to an arbitrary parametrization of the curve, 
the higher conformal curvatures can be expressed as 
\begin{align}
K_i=
\frac{\sqrt{\De_{i+1}\De_{i+3}}}{-\De_{i+2}\sqrt[4]{-\De_4}},
\qquad\text{for $i=2,\dots,n-1$}.
\label{Ki}
\end{align}
\end{theorem}

\begin{proof}
According to Lemma \ref{lem-gram}, the weight of the right hand side of \eqref{Ki} is 
\begin{align*}
\tfrac12(i+1)(i-2)+\tfrac12(i+3)i-(i+2)(i-1)-1=0,
\end{align*}
thus it defines an absolute conformal invariant.
With respect to the conformal arc-length parametrization, \eqref{Ki} coincides with \eqref{KiDe}, hence the statement follows.
\end{proof}

Alternatively, the statement is deducible from \eqref{KiDe} by the reparametrization according to subsection \ref{repa} and Lemma \ref{lem-gram}, where $g'=\Ph^{\frac14}$.
That way, we may obtain also an expression of the very first conformal invariant, $K_1$, which does not fit to the just given description:
Its expression with respect to the conformal arc-length is given in \eqref{U1'}, which transforms under reparametrizations according to \eqref{tiAA}.
The substitution of $g'=\Ph^{\frac14}$ and a small computation reveals that

\begin{proposition}
With respect to an arbitrary parametrization of the curve, 
the first conformal curvature can be expressed as
\begin{align}
\begin{split}
K_1
&= -\tfrac12 \Ph^{-\frac12} \left( \g{\dform U}{\dform U} - \tfrac12 (\ln \Ph)'' + \tfrac{1}{16} (\ln \Ph)'{}^2 \right) = \\
&=-\tfrac12\Ph^{-\frac52}\left( \Ph^2\,\g{\dform U}{\dform U}-\tfrac1{2}\Ph\Ph'' + \tfrac{9}{16}\Ph'^2 \right).
\label{K1}
\end{split}
\end{align}
\end{proposition}

\begin{remarks}\label{rem-K1}
(1)
Not only the conformal curvatures, but the tractor Frenet frame itself can be built with respect to an arbitrary parametrization.
E.g. expressions of  the first three tractors \eqref{U0} transform, under the reparametrization according to \eqref{Utilde} with $g'=\Ph^{\frac14}$, so that 
\begin{align}
\begin{split}
\form U_0 &= \Ph^{\frac14}\, \form T, \\
\form U_1 &= \form U + \tfrac14 (\ln \Ph)'\, \form T, \\
\form U_2 &= - \Ph^{-\frac14} \left( \dform U + \tfrac14 (\ln \Ph)'\, \form U + \left(\tfrac12 \g{\dform U}{\dform U} + \tfrac{1}{32} (\ln \Ph)'{}^2\right) \form T \right).  
\end{split}
\label{U0tilde}
\end{align}

(2)
Our exceptional invariant $K_1$ corresponds to a similarly exceptional invariant $J_{n-1}$ of \cite[section~5]{Fialkow1942} as follows.
For any relative conformal invariant $Q$ of weight 1, an appropriate combination of $Q$ and its derivatives leads to an appropriate transformation of the new quantity under reparametrizations of the curve.
Namely, with the same conventions as yet,
\begin{align*}
2\wt{Q}\wt{Q}''-3\wt{Q}'^2=g'^{-4}\left( 2QQ''-3Q'^2-2\cS(g) Q^2 \right).
\end{align*}
Combining with \eqref{tiAA}, it easily follows that the function 
\begin{align*}
Q^{-2}\,\g{\dform U}{\dform U}-2Q^{-3}Q''+3Q^{-4}Q'^2
\end{align*}
is an absolute conformal invariant of the curve.
Now, the substitution $Q=\Ph^{\frac14}$ leads to an invariant, which differs from our $K_1$, respectively Fialkow's $J_{n-1}$, just by a constant multiple.
(For the comparison with $J_{n-1}$, the formula \eqref{riemth} from the next subsection is needed.)
\end{remarks}

\subsection{Conformal curvatures in terms of Riemannian ones}\label{riemann}
Expanding any of the above expressions for the conformal curvatures yields very concrete and very ugly formulas in terms of the underlying derived vectors.
As a compromise between the explicitness and the ugliness we show how to rewrite the previous formulas in terms of  the Riemannian curvatures of the curve.
This way we will also be able to compare our invariants with their more classical counterparts.

Within this paragraph we consider a generic curve parametrized by  the Riemannian arc-length parameter, according to a chosen scale.
The corresponding Riemannian Frenet frame is denoted by $(e_1, \dots, e_n)$ and the Riemannian Frenet formulas are%
\footnote{Warning: in this subsection we relax the abstract indices, i.e. we write $e_i$ rather than $e_i^a$ etc. 
The raising and lowering indices with respect to the chosen metric is denoted by $\sharp$ and $\flat$, respectively.}
\begin{align*}
e_1'&=\ka_1 e_2,\\
e_i'&=-\ka_{i-1} e_{i-1}+\ka_i e_{i+1},\qquad\text{for}\ i=2,\dots,n-1,\\
e_n'&=-\ka_{n-1} e_{n-1},
\end{align*}
where primes denote the derivative with respect to the Riemannian arc-length and $\ka_1, \dots, \ka_{n-1}$ are the Riemannian curvatures of the curve.
Accordingly, the expressions \eqref{uu} and \eqref{U^a} read as
\begin{equation*}
u=1,\qquad U=e_1,\qquad U'=\ka_1 e_2,\qquad U''=-\ka_1^2 e_1+\ka_1' e_2+\ka_1\ka_2 e_3,\qquad \text{etc.}
\end{equation*}
Following the development of subsection \ref{lift}, the first tractors are 
\begin{equation*}
\begin{gathered}
\form T=\pmat{0\\0\\1},\qquad
\form U=\pmat{0\\e_1\\0},\qquad
\dform U=\pmat{-1\\\ka_1 e_2\\-\Rho(e_1,e_1)},\\
\ddform U=\pmat{0\\-\ka_1^2e_1+\ka_1'e_2+\ka_1\ka_2 e_3-\Rho(e_1,e_1)e_1-\Rho(e_1,-)^\sharp\\ *}.
\end{gathered} 
\end{equation*}
Hence we obtain
\begin{align}
\begin{split}
\g{\dform U}{\dform U}&=\ka_1^2+2\Rho(e_1,e_1),  \\
\Ph=\g{\ddform U}{\ddform U}-(\g{\dform U}{\dform U})^2&=\ka_1'^2+\ka_1^2\ka_2^2
+\text{terms involving $\Rho$}.
\label{riemth}
\end{split}
\end{align}

Substitution of \eqref{riemth} into \eqref{K1} leads to an expression of the first conformal curvature $K_1$ in terms of first two Riemannian curvatures $\ka_1$, $\ka_2$, some terms involving the Schouten tensor and their derivatives.
The full expansion leads to a huge formula even in the flat case.

Formula enthusiasts may continue further in this spirit and express the higher derived tractors.
We add some details on the Gram matrix corresponding to the first five tractors.
It is displayed in \eqref{starting-gram} where, modulo terms involving $\Rho$, 
\begin{equation*}
\begin{gathered}
\al=\ka_1^2, \qquad
\be=\ka_1'^2+\ka_1^4+\ka_1^2\ka_2^2, \\
\ga=9\ka_1^2\ka_1'^2+(\ka_1''-\ka_1^3-\ka_1\ka_2^2)^2+(2\ka_1'\ka_2+\ka_1\ka_2')^2 +\ka_1^2\ka_2^2\ka_3^2.
\end{gathered}
\end{equation*}
Of course, the term $\ka_3$ is nontrivial only if $n\ge4$.
From this and \eqref{Ki} one could deduce an expression for the second conformal curvature $K_2$ in terms of $\ka_1$, $\ka_2$, $\ka_3$, $\Rho$ and their derivatives.
Although the result is expected to be messy, it can be condensed into the following neat form:
\begin{align*}
K_2=\left(\ka_1'^2+\ka_1^2\ka_2^2\right)^{-\frac54} 
\left( (2\ka_1'^2\ka_2+\ka_1^2\ka_2^3+\ka_1\ka_1'\ka_2'-\ka_1\ka_1''\ka_2)^2 +(\ka_1'^2+\ka_1^2\ka_2^2)\ka_1^2\ka_2^2\ka_3^2 \right)^{\frac12},
\end{align*}
modulo terms involving $\Rho$.

\begin{remark}
In the flat case, for $n=3$, the previous display agrees precisely with the invariant which is called \textit{conformal torsion} in \cite{Cairns1994} (the other invariant there corresponds to our $K_1$).
Note that vanishing of $K_2$ in such case is equivalent to $\ka_2=0$, or $\ka_2\ne 0$ and
$\frac{\ka_2}{\ka_1}=\left( \frac{\ka_1'}{\ka_1^2\ka_2} \right)'$.
The former condition means that the curve is planar, the latter one means the curve is spherical.
By virtue of Remark \ref{rem-frenet}(2), this is an expected behaviour.
\end{remark}

For another comparisons let us also describe first few tractors from the tractor Frenet frame constructed in subsection \ref{frenet}.
According to \eqref{U0tilde}, the easy tractors  are
\begin{gather*}
\form U_0=\Ph^{\frac14}\pmat{0\\0\\1},\qquad
\form U_1=\pmat{0\\e_1 \\ \frac14(\ln{\Ph})'},\qquad
\form U_{2}=-\Ph^{-\frac14} \pmat{-1\\ \frac14(\ln\Ph)'e_1+\ka_1 e_2 \\ \frac1{32}(\ln\Ph)'^2+\frac12\ka_1^2},
\end{gather*} 
where $\Ph$ is given in \eqref{riemth}.
The first nontrivial step concerns an expression of $\form U_3$, 
which one calculates as
\begin{gather}
\form U_3=\Ph^{-\frac12}\pmat{0\\\ka_1'e_2+\ka_1\ka_2 e_3+\Rho(e_1,e_1)e_1-\Rho(e_1,-)^\sharp\\ \ka_1\ka_1'}.
\label{riemV3}
\end{gather} 

\begin{remarks}
(1)
Incidentally, the middle slot of the tractor in \eqref{riemV3} corresponds to the starting quantity in Fialkow's approach, the so-called \textit{first conformal normal} of the curve, cf. \cite[equation (4.21)]{Fialkow1942}.
Using this and an invariantly defined derivative along the curve (with respect to the conformal arc-length parameter), the conformal Frenet identities and the corresponding $n-2$ conformal curvatures are deduced in \cite[section~4]{Fialkow1942}.
The last Fialkow's conformal curvature is constructed ad hoc in \cite[section~5]{Fialkow1942}, see the discussion in Remark \ref{rem-K1}(2).

(2)
At this stage we can easily count the orders of individual invariants although some estimates could be done earlier.
The relative conformal invariant $\Ph$ is of order 3, cf. \eqref{riemth}.
From this and \eqref{K1} we see that the first conformal curvature $K_1$ is of order 5.
For $i\ge 2$, an estimate based on \eqref{Ki} yields the order of $i$th conformal curvature $K_i$ is at most $i+3$, but the current expressions show it is actually $i+2$.
\end{remarks}

\section{Null curves in general signature}\label{Indefinite}
The next step  is to consider possible generalizations of the previous treatment to 
conformal manifolds of arbitrary signature.
Thus, in the following we suppose that $M$ is a conformal manifold dimension $n=p+q$ and signature $(p,q)$ and $\T$ is the standard tractor bundle with the tractor metric of signature $(p+1,q+1)$.
Generally, we should distinguish curves according to the type of tangent vectors/tractors and also
according to the type of further tractors in the tractor Frenet frame.
This leads to a diversity of possible cases, cf. Remark \ref{indef4}(3).
The typical situation we need to understand is when the tangent vector together with its several derivatives generate a totally isotropic subspace of the tangent space, cf. the notion of $r$-null curves below.
One can view this as a counterpart of the Riemannian setting discussed so far. 
In the last subsection we discuss the Lorenzian signature in more detail.

\subsection{$r$-null curves}\label{null}
For null curves the density \eqref{uu} vanishes identically and cannot be used for the lift to $\T$.
Depending on the signature, we may consider  curves that are more and more isotropic, without being degenerate.
On a conformal manifold of signature $(p,q)$, for any $r\le\min\{p,q\}$, 
a curve is called the \textit{$r$-null curve} if the vectors $U^a,\dots,\dif{U}{r-1}^a$ are linearly independent and null, whereas $\dif{U}{r}^a$ is not null.
Consequently, it holds
\begin{equation}
U^{(i)}_c\dif{U}{j}^c=
\begin{cases}
0,&\qquad\text{for}\ i+j<2r, \\ 
(-1)^{r-i}U^{(r)}_c \dif{U}{r}^c,&\qquad\text{for}\ i+j=2r.
\end{cases}
\label{UiUj}
\end{equation}
Note that the notion of $r$-null curves is well defined as it does not depend on the chosen scale.
Indeed, one can easily verify by induction using \eqref{transU} that 
\begin{equation*}
\dif{\wh{U}}{i}^a = \dif{U}{i}^a \ \ \text{mod} \ \  
\langle U^a, \ldots, \dif{U}{i-1}^a \rangle, 
\end{equation*}
for all $1 \leq i \leq {2r}$.
In this vein, space- and time-like curves may be termed \textit{0-null curves}.
As an opposite extreme, we use the notation \textit{$\infty$-null curves} for curves, whose all vectors $\dif{U}{i}^a$ are isotropic; such curves are necessarily degenerate. 
The most prominent---and the most degenerate---instance of such curves are the null geodesics.

For an $r$-null curve, it follows from \eqref{transU} that the norm squared of $\dif{U}{r}^c$ is a nowhere vanishing density of conformal weight 2.
Thus,
\begin{equation*}
u:=\sqrt{|U^{(r)}_c\dif{U}{r}^c|}\in\E[1]
\end{equation*}
provides a lift $\form T\in\T$ as in \eqref{T}, to which we add the derived tractors $\form U,\dform U$ etc. as before.
Explicit expressions are very analogous to those in \eqref{U}, \eqref{A} etc. up to some shift.
For $i\le 2r$, we have
\begin{equation*}
\diform{U}{i}=\pmat{0\\u^{-1}\dif{U}{i}^a+\cdots\\$*$}, 
\end{equation*}
where the zero in the primary slot appears as $-u^{-1}U_c\dif{U}{i-1}^c=0$ and the dots in the middle slot denote terms of lower order.
In particular, we have $\g{\form T}{\form T}=\dots=\g{\diform{U}{r-1}}{\diform{U}{r-1}}=0$ and 
\begin{equation}
\vo:=\g{\diform{U}{r}}{\diform{U}{r}}= \pm 1. 
\label{pm}
\end{equation}
The first nontrivial quantity is $\g{\diform{U}{r+1}}{\diform{U}{r+1}}$.
In order to simplify expressions, we will use the notation 
\begin{equation}
\al:=\g{\diform{U}{r+1}}{\diform{U}{r+1}},\qquad
\be:=\g{\diform{U}{r+2}}{\diform{U}{r+2}},\qquad
\ga:=\g{\diform{U}{r+3}}{\diform{U}{r+3}},\qquad
\text{etc.}
\label{alefbet}
\end{equation}
The pattern of the initial relations remains the same as in \eqref{starting-gram} up to the signs concerning $\vo$ and a shift in the south-east direction.
After some computation, one shows that the first few antidiagonals of the Gram matrix read as
\begin{align}
\g{\diform{U}{i}}{\diform{U}{j}} = 
\begin{cases}
0,&\qquad\text{for}\ i+j<2r, \\
(-1)^{r-i}\,\vo,&\qquad\text{for}\ i+j=2r, \\
0,&\qquad\text{for}\ i+j=2r+1, \\
(-1)^{r+1-i}\, \al,&\qquad\text{for}\ i+j=2r+2, \\
(-1)^{r+1-i}\, \frac{2(r-i)+3}{2} \al',&\qquad\text{for}\ i+j=2r+3, \\
(-1)^{r-i}\, \be +(-1)^{r+1-i}\, \frac{(r-i+2)^2}{2} \al'',&\qquad\text{for}\ i+j=2r+4,
\end{cases}
\label{starting-anti}
\end{align}
where we use the convention $\diform{U}{-1}:=\form T$.

\smallskip
Under the reparametrization $\ti t=g(t)$ as in subsection \ref{repa}, one easily verifies that  
$\tform T=g'^{r+1}\form T$, from which one deduces the transformations of higher order tractors.
Considering $\diform{U}{i}:=0$, for $i<-1$, it turns out that 
\begin{equation}
\tdiform{U}{k}=g'^{r-k}\left( \diform{U}{k} +A_k g'^{-1}g''\diform{U}{k-1} +(B_k g'^{-1}g''' -C_k g'^{-2}g''^2)\diform{U}{k-2} + \cdots \right),
\label{ti-Uk}
\end{equation}
for $k=0,1,\dots$, where the coefficients $A_k,B_k,C_k$ are given by the recurrence relations
\begin{align*}
A_{k+1}&=A_k+r-k,\\
B_{k+1}&=B_k+A_k,\\
C_{k+1}&=C_k-(r-k-1)A_k,
\end{align*}
with the initial conditions $A_0=r+1$, $B_0=0$ and $C_0=0$.
It is an elementary, but tedious, exercise to solve and tidy up this system.
For $k=r+1$, it turns out that
\begin{align*}
A_{r+1}&=\tfrac12 (r+1)(r+2),\\
B_{r+1}&=\tfrac16 (r+1)(r+2)(2r+3),\\
C_{r+1}&=-\tfrac18 (r+1)(r+2)(r^2-r-4).
\end{align*}
From this and the initial relations  it follows that
\begin{align*}
\wt\al=g'^{-2}\left( \al -2\vo B_{r+1} g'^{-1}g''' +\vo (A_{r+1}^2 + 2C_{r+1})g'^{-2}g''^2 \right).
\end{align*}
Since $A_{r+1}^2 + 2C_{r+1}=3B_{r+1}$, the previous simplifies to
\begin{align}
\wt\al=g'^{-2}\left( \al - 2\vo B_{r+1}\cS(g) \right).
\label{ti-al}
\end{align}
Thus, the equation $\al=0$,
regarded as a condition on the parametrization of an $r$-null curve, determines a preferred family of parametrizations with freedom the projective group of the line.
This is a generalization of Proposition \ref{prop-repa} which just corresponds to $r=0$.
Altogether we may conclude with

\begin{proposition}\label{projparam}
For any admissible $r$, any $r$-null curve carries a preferred family of projective parameters.
\end{proposition}

It is well known that null geodesics carry a preferred family of affine parameters. 

\smallskip
Further, the Gram determinants are defined and denoted just the same as in \eqref{gram}.
Initial relations  for an $r$-null curve lead to $\De_i=0$, for $i=1,\dots,2r+2$, and $\De_{2r+3}= -\epsilon$.
The first potentially nontrivial determinant is $\De_{2r+4}$.
Analogously to Lemma \ref{lem-gram} we have

\begin{lemma}
Along any $r$-null curve, the Gram determinant $\De_i$, for $i=2r+4,\dots,n+2$, is a relative conformal invariant of weight $i(i-2r-3)$.
\end{lemma}

For any $i=2r+4,\dots,n+2$, the vanishing of $\De_i$ is equivalent to the fact that the determining tractors are linearly dependent.
The weight of any $\De_i$ is even, hence its sign is not changed by reparametrizations.
In contrast to Lemma \ref{Ph} (corresponding to $r=0$), we cannot say anything about the particular sign of $\De_{2r+4}$.
For an $r$-null curve with nowhere vanishing $\De_{2r+4}$, we define  the \textit{conformal pseudo-arc-length} parameter by integration of the 1-form 
\begin{equation}
ds:=\sqrt[2r+4]{|\De_{2r+4}(t)|}\,dt
\label{null-ds}
\end{equation}
along the curve; it is given uniquely up to an additive constant.
The invariant $\De_{2r+4}$  (hence also all higher ones) may vanish. 
In particular, this happens automatically if $2r+4>n+2$, i.e. if $n<2r+2$.
From the introductory counts we know that $r$ and $n$ are related by $n\ge 2r$, therefore  the critical dimensions are $n=2r$ and $2r+1$.
In such cases, none of the just considered invariants can be used to define a natural parameter of the curve.
However, there is another couple of relative conformal invariants as shown in subsection \ref{wilczynski} below; 
see also Remark \ref{dim3} for further comments.

\smallskip
At this stage, we sketch how to adapt the construction of the tractor Frenet frame from subsection \ref{frenet} for $r$-null curves with the pseudo-arc-length defined.
The smallest nondegenerate subbundle in $\T$ gradually built from the derived tractors is $\spanb{\form T,\form U,\dots,\diform{U}{2r+1}}$.
The restriction of the tractor metric to this subbundle has signature $(r+2,r+1)$, respectively $(r+1,r+2)$, thus its orthogonal complement in $\T$ has  signature $(p-r-1,q-r)$, respectively $(p-r,q-r-1)$.
We may always start with the prescription
\begin{equation}
\form U_0:=\form T,\qquad
\form U_1:=\form U,\qquad
\dots,\qquad
\form U_{r+1}:=\diform{U}{r}.
\label{null-U0}
\end{equation}
Then we have to consider transformations determining 
$\form U_i\in \spanb{\form T,\form U,\dots,\diform{U}{i-1}}$, for $i=r+2,\dots,2r+2$, so that the Gram matrix corresponding to the tractors $(\form U_0,\dots,\form U_{2r+2})$ 
is antidiagoanal with values $\pm 1$ (this substitutes the upper-left block in \eqref{grU}).
Note that in contrast to the case $r=0$, the definition for $\form U_{r+2},\dots,\form U_{2r+2}$ is not unique. 
This freedom is already visible from the first tractor of this subsequence which may be given as
\begin{equation}
\form U_{r+2}=-\vo \diform{U}{r+1} -\tfrac12\vo\al\diform{U}{r-1} +\cdots,
\label{null-U3}
\end{equation}
where the dots stay for any linear combination of lower order tractors $\diform{U}{r-2},\dots,\form T$.
To accomplish the full tractor frame, i.e. to determine $(\form U_{2r+3},\dots,\form U_{n+1})$, we continue by an orthonormalization process.
Going along, we may meet an isotropic tractor: in such case we face the same problem as discussed in Remark \ref{indef4}(3) and some additional choices are inevitable.

Having constructed the tractor Frenet frame, we differentiate with respect to the conformal pseudo-arc-length to obtain the tractor Frenet equations with the generating set of absolute conformal invariants.
Just from \eqref{null-U0} and \eqref{null-U3} we easily deduce first few Frenet-like identities:
\begin{equation*}
\begin{gathered}
\dform U_0=\form U_1,\qquad
\dots,\qquad
\dform U_r=\form U_{r+1},\\
\dform U_{r+1}=K_1\form U_{r} -\vo \form U_{r+2}, 
\qquad\text{where}\qquad 
K_1:=-\tfrac12\vo\al.
\end{gathered}
\end{equation*}
Clearly, the first conformal curvature is not affected by the freedom in the construction of the tractor Frenet frame. 
From the expression of $K_1$ and the remark preceeding Proposition \ref{projparam} we immediately have the following generalization of Proposition \ref{prop-K1}.

\begin{proposition}\label{null-prop-K1}
The conformal pseudo-arc-length parameter belongs to the projective family of parameters of an $r$-null curve if and only if $K_1=0$.
\end{proposition}

The freedom in the construction of $(\form U_{r+2},\dots,\form U_{2r+2})$ influences the corresponding part of tractor Frenet equations and so the corresponding conformal curvatures, namely, $K_2,\dots,K_{r+1}$.
The respective freedom in the construction of $(\form U_{2r+3},\dots,\form U_{n+1})$ influences some of  the remaining conformal curvatures $K_{r+2},\dots,K_{n-r-1}$.
Notably, they are $n-r-1$ in total.

Various expressions analogous to \eqref{Kie} and \eqref{KiVj} (with respect to the conformal pseudo-arc-length parameter), or \eqref{Ki} and \eqref{K1} (with respect to an arbitrary parameter) are deducible.
We provide more details only in subsection \ref{lorenz} for the case of Lorenzian signature.

\subsection{Invariants of Wilczynski type}\label{wilczynski}
For an $r$-null curve, vanishing of $\De_{2r+4}$ means that the tractors $\form T,\form U, \dots,\diform{U}{2r+2}$ are linearly dependent. 
Since the first $2r+3$ tractors from this sequence are linearly independent and $\diform{U}{k}=\diform{T}{k+1}$, this condition may be interpreted as a tractor linear ODE,
\begin{align}
\diform{T}{2r+3}+q_{2r+2}\diform{T}{2r+2}+\dots+q_1\dform T+q_0\form T=0,
\label{ode}
\end{align}
where $q_0, \dots,q_{2r+2}$ are functions expressible in terms of pairings of the initial tractors.
Following \cite[ch.~II, \S~4]{Wilczynski1906}, we may consider the Wilczynski invariants of this equation, i.e. the set of $2r+1$ essential invariants denoted as $\Th_3,\dots,\Th_{2r+3}$ (the indices indicate weights). 
Of course, we may associate Wilczynski invariants to the linear differential operator on the left hand side of \eqref{ode}, i.e. also to general $r$-null curves with nonzero $\De_{2r+4}$.
As a matter of fact, these invariants are relative conformal invariants of the curve. 

Assuming $q_{2r+2}=0$, i.e.\ the so-called  semi-canonical form, an explicit form of first two of these invariants is given by the following Lemma.

\begin{lemma} 
\label{cool}
Let $y^{(k)}+q_{k-2}y^{(k-2)}+\dots+q_0 y=0$ be a linear ODE of $k$th order in the semi-canonical form ($q_{k-1}=0$).
Then the first two Wilczynski invariants are
\begin{align*} 
& \Th_3=q_{k-3}-\frac{k-2}2 q_{k-2}', \\
& \Th_4 =q_{k-4} -\frac{k-3}2q'_{k-3} +\frac{(k-2)(k-3)}{10} q_{k-2}'' -\frac{(5k+7)(k-2)(k-3)}{10k(k+1)(k-1)} q_{k-2}^2.
\end{align*}
\end{lemma}

The form of $\Th_3$ is obtained already by Wilczynski, whereas it seems the formula for $\Th_4$ did not appear in the literature so far. 
It requires extremely tedious computation using Wilczynski's approach.
We have used an alternative technique of the so-called quantization, cf. \cite{OT}.
Details will be published elsewhere. 

\smallskip

Now we need to relate coefficients $q_i$ from the equation \eqref{ode} to quantities $\al$, $\be$, $\ga$ etc. from \eqref{alefbet}.
By pairing both sides of \eqref{ode} with $\form T$, it follows from \eqref{starting-anti} that $q_{2r+2}=0$, i.e.\ the equation is in the semi-canonical form.
Considering similarly the pairing of both sides of \eqref{ode} with $\dform{T} = \form{U}$,  $\ddform{T} = \dform{U}$ and with $\dddform{T} = \ddform{U}$, one computes
\begin{align} \label{q}
q_{2r+2}=0,\qquad q_{2r+1} = \vo \al,\qquad q_{2r} = \frac{2r+1}2\vo \al', \qquad 
q_{2r-1} = -\vo \be  + \frac{r^2}{2}\vo \al''+ \al^2.
\end{align}
Combining this with Lemma \ref{cool}, we obtain  the resulting form of $\Th_4$:

\begin{proposition}\label{wilca}
Let $\vo$, $\al$ and $\be$ be the quantities associated to an $r$-null curves (with respect to an arbitrary parametrization) as in \eqref{pm} and \eqref{alefbet}.
Then $\Th_3 = \Th_5 =0$ and 
\begin{align*}
\Th_4 = -\vo \be -\frac{r(r+3)}{10}\vo \al'' + \frac{(r+3)(2r+5)(5r+4)}{10(r+1)(r+2)(2r+3)} \al^2
\end{align*} 
is a relative conformal invariant of the weight 4.
\end{proposition}

\begin{proof}
Both vanishing of $\Th_3$ and the form of $\Th_4$ follows from \eqref{q} and Lemma \ref{cool}, where $k=2r+3$.
The particular weight of $\Th_4$ follows from the definitions \eqref{alefbet} and the exponent of $g'$ in \eqref{ti-Uk}.

To show $\Th_5 =0$, we shall for simplicity assume a projective parametrization of the curve, i.e.\ $\al=0$. 
In particular, the equation is in the canonical Laguerre--Forsyth form for which the expressions of Wilczynski invariants are well known, see \cite[ch.~II, eqn.~(48)]{Wilczynski1906}. 
Analogously to the computation leading to \eqref{q}, we obtain
$$
\g{\diform{U}{i}}{\diform{U}{2r+5-i}} =   (-1)^{r-i} \frac{2(r-i)+5}{2}\be'
\qquad \text{and} \qquad
q_{2r-2} = -\frac{2r-1}{2}\vo \be'.
$$
Now substituting coefficients $q_{2r-1}$ and $q_{2r-2}$ into the just referred formula, 
one easily verifies that  $\Th_5=0$.
\end{proof}

Observe that $\Th_4$  recovers $\De_4$ from \eqref{De4} for $r=0$.

\begin{remarks}
\label{rem-wilca}
(1)
The Laguerre--Forsyth form of the equation is preserved by transformations with freedom the projective group of the line, which corresponds precisely to the condition $\al=0$.
Also, it follows that the linear equation is equivalent to the trivial equation if and only if all Wilczynski invariants vanish.
In our case, this is equivalent to the vanishing of $\al$, $\be$, etc. from \eqref{alefbet} up to $\g{\diform{U}{2r+2}}{\diform{U}{2r+2}}$.

(2)
Experiments with specific $r$'s suggest the equation \eqref{ode} could be self-adjoint, i.e. all odd Wilczynski invariants vanish. 
This would mean that we have just $r$ (rather than $2r$) new and potentially nontrivial invariants 
$\Th_4,\dots,\Th_{2r+2}$.
\end{remarks}

\subsection{Conformal null helices}\label{null-geod}
For any space- or time-like curve with an arbitrary parametrization, a metric in the conformal class may be chosen so that 
$U'^a=0$, cf. \eqref{transacc}.
For null curves, this condition holds only if the curve is \textit{null geodesic}, i.e.
the null curve whose acceleration vector $U'^a$ is proportional to $U^a$.
Consequently, all higher order vectors are proportional to $U^a$.
These curves, although very important, cannot be lifted to the standard tractor bundle in the sense considered above. 
That is why they are excluded from our considerations.

For an $r$-null curve, we see from \eqref{UiUj} and the subsequent discussion that the vectors $U^a,\dots,\dif{U}{2r}^a$ are linearly independent for any choice of scale.
Therefore we cannot achieve any of the conditions $U'^a=0$, \dots, $\dif{U}{2r}^a=0$ by a conformal change of metric.
The simplest conceivable statement is 

\begin{lemma} \label{'''}
For an $r$-null curve with an arbitrary parametrization, a metric in the conformal class may be chosen so that $U^{(r)}_c U^{(r)}{}^c=\pm 1$ and  $\dif{U}{2r+1}^a=0$.
\end{lemma}

\begin{proof}
This can be proved by various means.
For the sake of latter references, we employ the tractors anyhow it may seem artificial. 
Firstly, we may always choose the metric so that $U^{(r)}_c U^{(r)}{}^c=\pm 1$.
With this assumption, the tractors associated to the curve have the form
\begin{equation*}
\begin{gathered}
\form T=\pmat{0\\0\\1}, \qquad
\form U=\pmat{0\\U^a\\0}, \qquad
\dform U=\pmat{0\\U'^a  \\ \rho_1}, \qquad
\dots,\\
\diform{U}{j}=\pmat{\ep_j \\ \dif{U}{j}^a \!\!\!\mod \langle U^a, \ldots,\dif{U}{j-2}^a \rangle \\ \rho_j},
\end{gathered}
\end{equation*}
for $j = 2, \ldots, 2r+1$, where $\rho_j$ are in general nonzero while all $\ep_j$ vanish except for $\ep_{2r+1}=(-1)^{r-1}$.

Secondly, the choice of metric may be adjusted so that the bottom slot of $\dform U$ vanishes and the expressions of previous tractors are unchanged:
according to \eqref{transtr}, the corresponding $\Up_a$ has to satisfy $\Up_cU'^c=\rho_1$ and $\Up_c U^c=0$.
Hence 
$\ddform U=\smat{0\\ U''^a \\ \rho_2}$, 
i.e. there is only the leading term in the middle slot.
Inductively, we may achieve a rescaling so that  
$\diform{U}{j}=\smat{0 \\ \dif{U}{j}^a  \\ 0}$, for all $j\le 2r$.
Hence 
$\diform{U}{2r+1}=\smat{(-1)^{r-1} \\ \dif{U}{2r+1}^a  \\ \ast}$.
Finally, we may consider a  rescaling so that the middle slot of the last tractor vanishes:
according to \eqref{transtr}, the corresponding $\Up^a$ equals to $(-1)^r\dif{U}{2r+1}^a$ along the curve.
Hence the statement follows.
\end{proof}

Now we are ready to identify the closest relatives of conformal circles among $r$-null curves.
We know from preceding subsections that the first $2r+3$ derived tractors $\form T,\form U,\dots,\diform{U}{2r+1}$ are linearly independent.
Therefore we can never achieve any of the conditions $\ddform U=0$, \dots, $\diform{U}{2r+1}=0$ for $r$-null curves.
The simplest conceivable condition appears in item (a) of the following Theorem. 

\begin{theorem}
\label{r-null}
For an $r$-null curve, the following conditions are equivalent:
\begin{enumerate}[(a)]
\item the curve with a projective parametrization $(\al=0)$ obeys $\diform{U}{2r+2}=0$,
\item the curve (with and arbitrary parametrization) obeys $\De_{2r+4}=0$ and $\Th_4=\dots=\Th_{2r+3}=0$,
\item there is a metric in the conformal class such that $U^{(r)}_c U^{(r)}{}^c=\pm 1$, $\dif{U}{2r+1}^a=0$ and $U^c\Rho_{ca}=0$.
\end{enumerate}
\end{theorem}

\begin{proof}
  Expressing the condition $\De_{2r+4}=0$ as \eqref{ode}, the equivalence of (a) and (b) follows from the general theory of linear differential equations and the fact that $\Th_3$ vanishes automatically, cf. Proposition \ref{wilca} and Remark \ref{rem-wilca}(1).

For the rest we assume the scale is chosen as in the proof of Lemma \ref{'''}.
In such scale, the derived tractors are
\begin{equation*}
\begin{gathered}
\form T=\pmat{0\\0\\1},\quad
\form U=\pmat{0\\U^a\\0},\quad
\dots,\quad
\diform{U}{2r}=\pmat{0\\\dif{U}{2r}^a\\0},\quad
\diform{U}{2r+1}=\pmat{(-1)^{r-1}\\0\\-U^c\dif{U}{2r}^d\Rho_{cd}}.
\end{gathered}
\end{equation*}
The additional assumption $U^c\Rho_{ca}=0$ from (c) then yields $\diform{U}{2r+2}=0$ and so (a) holds.

Conversely, from (a) it in particular follows $\g{\diform{U}{2r+1}}{\diform{U}{2r+1}}=0$, hence the bottom slot of $\diform{U}{2r+1}$ has to vanish.
Then $\diform{U}{2r+2}=0$ implies that $U^c\Rho_{ca}=0$ and so  (c) holds.
\end{proof}

Any curve satisfying (any of) the conditions (a)--(c) is called \textit{conformal $r$-null helix}.
The condition $\De_{2r+4}=0$ is equivalent to the fact that the subbundle $\spanb{\form T,\form U,\dots,\diform{U}{2r+1}}\subset\T$ of rank $2r+3$ is parallel along the curve.
In the flat case, this means that the conformal $r$-null helix is contained in the conformal image of a $(2r+1)$-dimensional Euclidean subspace. 
With condition (c), it easily follows that these curves are just the conformal images of so-called \textit{null Cartan helices}, cf. \cite{Duggal2007}.
Note that the conformal circles fit to the just given description for $r=0$.

\begin{remark}
As in subsection \ref{prop-geod}, one may construct conserved quantities along conformal $r$-null helices on specific conformal manifolds.
The two examples mentioned in Propositions \ref{prop-ein} and \ref{prop-ckf} 
have obvious counterparts for general $r$, provided that the respective functions are replaced by 
\begin{equation*}
\frak s:=\g{\diform{U}{2r+1}}{\form S}
\qquad\text{and}\qquad
\frak k:=\form K(\form U,\diform{U}{2r+1}).
\end{equation*}
The reasoning is the same, explicit expressions are analogous, but more complicated.
\end{remark}

\subsection{Remarks on Lorenzian signature}\label{lorenz}
In this subsection we suppose an $n$-dimensional conformal manifold of Lorenzian signature $(n-1,1)$, hence the standard tractor bundle $\T$ with the bundle metric of signature $(n,2)$.
Following the scheme of previous subsections, a lot of things simplify as the maximal isotropic subspace in the tangent space has dimension one.

The only admissible $r$-null curves in this signature correspond to $r=1$.
Therefore we may speak without a risk of confusion just about null curves, instead of 1-null curves.
The assumption $U_cU^c=0$ implies that $U_cU'^c=0$, thus the vector $U'^a$ must be space-like.
The lift to $\T$ is provided by the density $u=\sqrt{U'_c U'^c}$.
The smallest nondegenerate subbundle in $\T$ built from the derived tractors is of rank 5 and has signature $(3,2)$.
Hence its orthogonal complement is positive definite.
The first few starting relations are indicated in the following Gram matrix for the sequence 
$\big(\form T,\form U,\dform U,\ddform U,\dddform U, \ddddform U \big)$, cf. \eqref{alefbet} and \eqref{starting-anti}:
\begin{equation*}
\pmat{0&0&0&0&1&0\\[0.5ex]
0&0&0&-1&0&\al\\[0.5ex]
0&0&1&0&-\al&-\tfrac32 \al'\\[0.5ex]
0&-1&0&\al&\tfrac12 \al'&\tfrac12 \al''-\be\\[0.5ex]
1&0&-\al&\tfrac12 \al'&\be&\tfrac12 \be'\\[0.5ex]
0&\al&-\tfrac32 \al'&\tfrac12 \al''-\be&\tfrac12 \be'& \ga}.
\end{equation*}

The projective parametrization of the null curve corresponds to $\al=\g{\ddform U}{\ddform U}=0$.

The trivial Gram determinants are $\De_1=\De_2=\De_3=\De_4=0$ and $\De_5=1$.
The first nontrivial one is
\begin{align*}
\De_6
&=\ga+\al''\al-2\be\al-\tfrac94\al'^2+\al^3,
\end{align*}
provided that the dimension of conformal manifold is $n\ge 4$, which is assumed in the rest of this subsection (the special case $n=3$ is dealt individually in Remark \ref{dim3}).
This invariant is nonnegative:

\begin{lemma}
$\De_6\ge0$.
\end{lemma}
\begin{proof}
For a projective parametrization of the null curve, $\al=0$, we have $\De_6=\ga$.
According to this choice, one easily verifies that the tractor
\begin{equation*}
\form V_5 := \ddddform U - \be\,\form U -  \tfrac12\be'\,\form T
\end{equation*}
belongs to the orthogonal complement of the subbundle $\spanb{\form T,\form U,\dots,\dddform U}$ in $\T$. 
This complement has the positive definite signature, 
hence we conclude with 
$\ga=\g{\ddddform U}{\ddddform U} = \g{\form V_5}{\form V_5} \geq 0$.
Since the weight of $\De_6$ is even, it is a nonnegative function independently of the parametrization of the curve.
\end{proof}

Accordingly, the 1-form \eqref{null-ds} leading to the definition of conformal pseudo-arc-length may be substituted by 
\begin{equation}
ds:=\sqrt[6]{\De_6(t)}\,dt.
\label{lorenz-ds}
\end{equation}

Clearly, vanishing of $\De_6$ is equivalent to vanishing of $\form V_5$, which yields exactly the equation \eqref{ode}, for $r=1$.
Among the corresponding Wilczynski invariants, there is only one nontrivial, namely,
\begin{align}
\begin{split}
\Th_4
&=-\tfrac1{25}(25\be+10\al''-21\al^2),
\end{split}
\label{lorenz-Th4}
\end{align}
cf.\ Proposition \ref{wilca}.
According to Theorem \ref{r-null}, we may conclude with

\begin{proposition}
\label{1-null}
Conformal null helices are characterized by the pair of equations 
\begin{equation*}
\De_6=0
\qquad\text{and}\qquad
\Th_4=0.
\end{equation*}
\end{proposition}

To build the tractor Frenet frame, we start with 
\begin{equation}
\form U_0:=\form T,\qquad
\form U_1:=\form U,\qquad
\form U_{2}:=\dform U.
\label{lorenz-U0}
\end{equation}
Now, all admissible transformations determining $\form U_3$ and $\form U_4$, so that the 5-tuple $(\form U_0,\dots\form U_4)$ has the antidiagonal Gram matrix, are 
\begin{align}
\begin{split}
\form U_{3}&:=-\ddform U-\tfrac12\al\,\form U +k\,\form T, \\
\form U_{4}&:=\dddform U+\al\,\dform U +\ell \,\form U+\tfrac12(\al^2-\be)\,\form T,
\end{split}
\label{lorenz-U3}
\end{align}
where $k+\ell=\frac12\al'$.
For the rest of this demonstration we choose
\begin{equation}
k=0\qquad\text{and}\qquad \ell=\tfrac12\al'.
\label{kl}
\end{equation}
As mentioned  above, the orthocomplement to the subbundle $\spanb{\form T, \dots,\dddform U}=\spanb{\form U_0,\dots,\form U_4}$ has positive definite signature.
Therefore, we easily complete the tractor Frenet frame 
\begin{equation*}
(\form U_0, \dots, \form U_4; \form U_5, \dots,\form U_{n+1})
\end{equation*}
in the spirit of \eqref{Ui} so that $\form U_i\in \spanb{\form T,\form U,\dots,\diform{U}{i-1}}$, for all admissible $i$, and the corresponding Gram matrix is 
\begin{equation*}
\mbox{\footnotesize{$
\left(
\begin{BMAT}(@)[3pt]{ccc0ccc}{ccc0ccc}
&&1&&& \\ 
&\text{\reflectbox{$\ddots$}}&&&& \\ 
1&&&&& \\
&&&1&& \\
&&&&\ddots& \\
&&&&&1 \\
\end{BMAT}
\right).
$}}
\end{equation*}

Now, differentiating with respect to the conformal pseudo-arc-length we approach the identities with the generating set of absolute conformal invariants.
Just from \eqref{lorenz-U0} and \eqref{lorenz-U3}, with respect to the choice \eqref{kl}, we easily deduce first four Frenet-like identities:
\begin{align}
\dform U_0&=\form U_1, \label{lorenz-U0'} \\
\dform U_1&=\form U_2, \\
\dform U_2&=K_1\form U_1-\form U_3, 
\qquad\text{where}\qquad 
K_1:=-\tfrac12\al, \label{lorenz-K1} \\
\dform U_3&=K_2\form U_0-K_1\form U_2-\form U_4,
\qquad\text{where}\qquad
K_2:=\tfrac12(\al^2-\be). \label{lorenz-K2}
\end{align}
The next step is subtle, but fully analogous to the one announced in Lemma \ref{lem-frenet}:
the only coefficient in the expression of $\dform U_4$, which cannot be deduced from the previous identities, is the one of $\form U_5$, i.e. $\g{\dform U_4}{\form U_5}$.
It however follows that this coefficient is constant, namely,
\begin{align}
\dform U_4=-K_2\form U_1  +\form U_5.
\label{lorenz-U4'}
\end{align}
(The crucial point in the argument is the expression of $\form V_5$ and the observation that $\g{\form V_5}{\form V_5}=\De_6$.
Since we assume parametrization by the conformal pseudo-arc-length, this equals to 1, thus $\form U_5=\form V_5$.)
The rest is easy so that one quickly summarizes as
\begin{align}
\begin{split}
&\dform U_5=-\form U_0+K_3\form U_6,\\
&\dform U_i=-K_{i-3}\form U_{i-1}+K_{i-2}\form U_{i+1},\qquad\text{for}\ i=6,\dots,n,\\
&\dform U_{n+1}=-K_{n-2}\form U_{n}.
\label{lorenz-trfrfor}
\end{split}
\end{align}
The cluster of equations from \eqref{lorenz-U0'} through \eqref{lorenz-U4'} to \eqref{lorenz-trfrfor} forms the {tractor Frenet equations} associated to  the null curve determining its conformal curvatures.
Schematically, the tractor Frenet equations may be written as
\begin{equation*}
\mbox{\footnotesize{$
\left(
\begin{BMAT}(@)[1.2pt]{c}{ccccc0cccc}
\dform U_0\\ \dform U_1\\ \dform U_2\\ \dform U_3\\ \dform U_4\\ \dform U_5\\ \dform U_6\\ \vdots\\  \dform U_{n+1}
\end{BMAT}
\right)
=
\left(
\begin{BMAT}(@)[1.8pt]{ccccc0cccc}{ccccc0ccccc}
&1&&&&&&& \\
&&1&&&&&& \\
&K_1&&-1&&&&& \\
K_2&&-K_1&&-1&&&& \\
&-K_2&&&&1&&& \\
-1&&&&&&\ \ K_3&& \\
&&&&&-K_3&&& \\
&&&&&&\ddots&\ \ \ \ddots&\\ 
&&&&&&&&K_{n-2} \\
&&&&&&&-K_{n-2}& \\
\end{BMAT}
\right)
\cdot
\left(
\begin{BMAT}(@)[2.1pt]{c}{ccccc0cccc}
\form U_0\\ \form U_1\\ \form U_2\\ \form U_3\\ \form U_4\\ \form U_5\\ \form U_6\\ \vdots\\ \form U_{n+1}
\end{BMAT}
\right)
$}}
\end{equation*}

Altogether, we summarize as follows:

\begin{proposition}
The system of tractor Frenet equations determines a generating set of absolute conformal invariants of the null curve.
On a conformal manifold of Lorenzian signature and dimension $n\ge 4$, it consists of $n-2$ conformal curvatures that are expressed,
with respect to the conformal pseudo-arc-length parametrization, 
as
\begin{align*}
K_i=
\begin{cases}
\g{\dform U_{i+1}}{\form U_{i+2}},&\qquad\text{for}\ i=1,2 \\
\g{\dform U_{i+2}}{\form U_{i+3}},&\qquad\text{for}\ i=3,\dots,n-2. 
\end{cases}
\end{align*}
Among these curvatures, only $K_2$ depends on additional choices in the construction of  $\form U_3$ and $\form U_4$.
\end{proposition}

As in subsection \ref{revise}, there are alternative ways of expressing the conformal curvatures.
The first two, i.e. the two exceptional, conformal curvatures $K_1$ and $K_2$ are given in terms of initial tractors with respect to the conformal pseudo-arc-length parametrization in \eqref{lorenz-K1} and \eqref{lorenz-K2}.
One may obtain the expressions with respect to an arbitrary reparametrization by a substitution according to \eqref{ti-al}, respectively \eqref{ti-Uk}, where $g'=\sqrt[6]{\De_6}$.
In the former case, one ends up with the formula very similar to the one in \eqref{K1}, the later case is more ugly.
For the higher conformal curvatures we may proceed as follows.
Firstly, an analogue of Proposition \ref{prop-revise} turns out to be 

\begin{proposition}
With respect to the conformal arc-length parametrization, 
the higher conformal curvatures can be expressed as 
\begin{align}
K_i =\sqrt{\frac{\g{\form V_{i+3}}{\form V_{i+3}}}{\g{\form V_{i+2}}{\form V_{i+2}}}},
\qquad\text{for $i=3,\dots,n-2$}.
\label{lorenz-KiVj}
\end{align}
\end{proposition}

Secondly, it holds $\De_{i+1}=\De_i (\g{\form V_i}{\form V_i})$, for all admissible $i$.
Since $\De_5=1$, we in particular see that all admissible determinants are positive.
Now, expressing $\g{\form V_i}{\form V_i}$ via the determinants, substituting into \eqref{lorenz-KiVj} and passing to an arbitrary parametrization, we obtain the following substitute of Theorem \ref{husty-thm}:

\begin{theorem} \label{lorenz-husty-thm}
With respect to an arbitrary parametrization,
the higher conformal curvatures can be expressed as 
\begin{align*}
K_i =\frac{\sqrt{\De_{i+2}\De_{i+4}}}{\De_{i+3}\sqrt[6]{\De_6}},
\qquad\text{for $i=3,\dots,n-2$}.
\end{align*}
\end{theorem}

\begin{remark}\label{dim3}
As we already noticed, the study of null curves in dimension $n=3$ requires an extra care.
The relative invariant $\De_6$ vanishes automatically in that case, therefore it cannot be used to define the distinguished parametrization of the curve.
However, we still have the Wilczynski invariant $\Th_4$ given by \eqref{lorenz-Th4}.
Thus, for a generic null curve in 3-dimensional manifold, we define an alternative \textit{conformal pseudo-arc-length} parameter  by integrating the 1-form
\begin{align*}
ds:=\sqrt[4]{|\Th_4(t)|}\,dt,
\end{align*}
instead of \eqref{lorenz-ds}.
Accordingly, we may follow the construction of conformal invariants as above with the following conclusion:
the whole tractor Frenet frame consists just from the first five tractors displayed in \eqref{lorenz-U0}--\eqref{lorenz-U3} and the corresponding tractor Frenet equations are just \eqref{lorenz-U0'}--\eqref{lorenz-U4'}, with $\form U_5=0$. 
In those equations, two conformal invariants appear, namely, $K_1$ and $K_2$.
It however follows from their expressions, and the fact that $\Th_4=1$ for the current conformal pseudo-arc-length parametrization, that $K_2$ is expressible as a function of $K_1$ and its derivatives, namely,
\begin{align*}
K_2 = -\frac{2}{5} K_1'' + \frac{8}{25} K_1^2+ \frac{1}{2}.
\end{align*}
Therefore we end up with only one significant absolute conformal invariant, $K_1$, as expected.
An interpretation of that invariant is still the same as in Propositions \ref{null-prop-K1}.

In the case that $\Th_4$ vanishes identically, we recover the conformal null helices discussed above, cf. Proposition \ref{1-null}.
\end{remark}

\subsection*{Acknowledgements}
Authors thank to Boris Doubrov, Rod Gover and Igor Zelenko for useful discussions.
Experiments, checks and comparisons of some explicit expressions were done with the computational system \textsc{Maple}.
J\v{S} was supported by the Czech Science Foundation (GA\v{C}R) under grant P201/12/G028, V\v{Z} was supported by the same foundation under grant GA17-01171S.

\end{document}